\documentclass[11pt]{amsart}

\usepackage{amssymb} 
\usepackage[latin1]{inputenc}
\usepackage[matrix,arrow,curve]{xy}
\usepackage[mathscr]{eucal}
\usepackage{epic,eepic}
\usepackage{bbm}

\makeindex

\newtheorem{theorem}{Theorem}
\newtheorem{lemma}[theorem]{Lemma}
\newtheorem{prop}[theorem]{Proposition}

\theoremstyle{definition}
\newtheorem{definition}[theorem]{Definition}
\newtheorem{remark}[theorem]{Remark}

 \newtheoremstyle{numero}
{1ex}
{1ex}
{\it}
{1cm}
{\scshape}
{}
{4pt}
{}
\theoremstyle{numero}
\newtheorem{num}{}[section]

\newtheorem{sousnum}{}[num]

\newenvironment{pf}
{\medskip\noindent {\it Proof --- \ }}
{\hfill\nobreak $\Box$ \par\bigbreak}

\newcommand{\mm}{\mathfrak{m}}

\newcommand{\Q}{{\bf{Q}}}
\newcommand{\Z}{{\bf{Z}}}
\newcommand{\C}{\bf{C}}

\newcommand{\End}{\mathrm{End}}

\newcommand{\Frob}{{\rm{Frob}}}

\newcommand{\GL}{\mathrm{GL}}

\newcommand{\Gl}{\mathrm{GL}}

\newcommand{\Gal}{\mathrm{Gal}}

\newcommand{\F}{{\bf{F}}}

\newcommand{\N}{{\bf{N}}}

\newcommand{\SL}{{\rm SL}}
\newcommand{\PGL}{\rm PGL}
\newcommand{\PSL}{\rm PSL}

\newcommand{\tr}{\mathrm{tr}\,}

\newcommand{\mat}[1]{ \left( \begin{matrix} #1 \end{matrix} \right)}

\newcommand{\rhob}{{\bar \rho}}
\renewcommand{\sf}{{\text{sf}}}

\newcommand{\Sc}{{\mathcal S}}
\newcommand{\Zc}{{\mathcal Z}}
\newcommand{\Uc}{{\mathcal U}}
\newcommand{\Nc}{{\mathcal N}}
\newcommand{\Fc}{{\mathcal F}}
\newcommand{\Bc}{{\mathcal B}}
\renewcommand{\Mc}{{\mathcal M}}
\renewcommand{\sf}{{\text{sf}}}
\date{}

\title{The number of non-zero coefficients of modular forms $\pmod p$}

\author{Joël BELLAÏCHE\,\and \,Kannan SOUNDARARAJAN}
\thanks{Joël Bellaïche was supported by NSF grant DMS 1101615. Kannan Soundararajan was partially supported by NSF grant DMS 1001068, and a Simons Investigator grant from the Simons Foundation.}
\address{Department of Mathematics, MS 050, Brandeis University, 415 South Street, Waltham, MA 02453} 
\email{jbellaic@brandeis.edu}
\address{Department of Mathematics, Stanford University, Stanford, CA 94305} 
\email{ksound@stanford.edu} 
\begin{document}
\baselineskip 17pt
\begin{abstract} Let $f = \sum_{n=0}^\infty a_n q^n$ be a modular form modulo a prime $p$, and let $\pi(f,x)$ be the number of non-zero coefficients $a_n$ for $n < x$. We give an asymptotic formula for $\pi(f,x)$; namely, if $f$ is not constant,
$$\pi(f,x) \sim   c(f) \frac{x}{(\log x)^{\alpha(f)}}  (\log \log x)^{h(f)},$$
where $\alpha(f)$ is a rational number such that $0 <\alpha(f) \leq 3/4$, $h(f)$ is a non-negative integer and $c(f)$ is a positive 
real number. We also discuss the question of the equidistribution of the non-zero values of the coefficients $a_n$.
\end{abstract}

\maketitle

\section{Introduction}

\noindent Let $f = \sum_{n=0}^\infty a_n q^n$ be a holomorphic modular form of integral weight $k \geq 0$ and some level $\Gamma_1(N)$ such that the coefficients $a_n$ are integers. Let $p$ be a prime number. Serre \cite{SerreE} has shown that the sequence $a_n \pmod{p}$ is {\it lacunary}.  That is, the natural density of the set of integers $n$ such that $p \nmid a_n$ is $0$. 
More precisely, Serre gives the asymptotic upper bound
\begin{eqnarray} \label{serrebound} 
| \{ n<x, a_n \not \equiv 0\pmod{p} \}  | \ll \frac{x}{(\log x)^\beta},
\end{eqnarray}
where $\beta$ is a positive constant depending on $f$. 
Later, Ahlgren  \cite[Lemma 2.1]{ahlgren} established the following asymptotic lower bound:  Assume that $p$ is odd, 
and that there exists an integer $n \geq 2$ divisible by at least one prime $\ell$ not dividing $Np$ such that $p \nmid a_n$.  Then 
\begin{eqnarray} \label{ahlgrenbound} 
| \{ n<x, a_n \not \equiv 0\pmod{p} \} | \gg \frac{x}{(\log x)}.
\end{eqnarray}
Under the same hypothesis, this lower bound was recently improved by Chen (\cite{chen}): for every $K \geq 0$
 \begin{eqnarray} \label{chenbound} 
| \{ n<x, a_n \not \equiv 0\pmod{p} \} | \gg \frac{x}{(\log x)} (\log \log x)^K,
\end{eqnarray}
where the implicit constant depends on $K$. 

In this paper, we improve on these results (\ref{serrebound}), (\ref{ahlgrenbound}) and (\ref{chenbound})
by giving an asymptotic formula for $|\{ n<x, a_n \not \equiv 0\pmod{p} \}|.$ To describe our results, we slightly change our setting by working directly with modular forms over a finite field, which allows for more generality and more flexibility.

Let $p$ be an odd prime,\footnote{When $p=2$, similar but slightly different results may be obtained, cf. \cite{BN}.}  and
$N \geq 1$ an integer. We define the space of modular forms of level $\Gamma_1(N)$ with coefficients in $\F_p$, denoted by $M(N,\F_p)$, as the subspace of $\F_p[[q]]$ generated by the reductions modulo $p$ of the $q$-expansions at $\infty$ of all holomorphic modular forms of level $\Gamma_1(N)$ and some integral weight $k \geq 0$ with coefficients in $\Z$. For $\F$ a finite extension of $\F_p$, we define $M(N,\F)$ as $M(N,\F_p) \otimes_{\F_p} \F$.  Given $f$ in $M(N,\F)$ let 
$$ 
{\pi}(f,x) = | \{ n<x: a_n \neq 0\}|. 
$$ 


\begin{theorem} \label{main} Let $f=\sum_{n=0}^\infty a_n q^n \in M(N,\F)$, and assume that $f$ is not  constant; that is,  $a_n \neq 0$ for some $n \geq 1$.   Then there exists a rational number $\alpha(f)$ with $0<\alpha(f) \leq 3/4$, an integer $h(f) \geq 0$, and a positive real constant $c(f) > 0$, such that
$$
{\pi}(f,x) \sim c(f) \frac{x}{(\log x)^{\alpha(f)}}  (\log \log x)^{h(f)}.
$$
\end{theorem} 

When $f$ is an {\sl eigenform} for all Hecke operators $T_m$ (that is $T_m f = \lambda_m f$, $\lambda_m \in \F$),
this theorem was established by Serre \cite{SerreE}, and in this case one has $h(f)=0$.  
 However, the case of eigenforms is special because, as shown by Atkin, Serre, Tate and Jochnowitz in the seventies, there are only finitely many normalized eigenforms in the infinite dimensional space $M(N,\F)$. 
One can decompose every $f \in M(N,\F)$ as a finite sum $\sum_i f_i$ of 
{\sl generalized eigenforms\footnote{We call  a form $f \in M(N,\F)$ a generalized eigenform $f$ if for every $\ell$ not dividing $Np$, there exists  $\lambda_\ell \in \F$ and $n_l \in \N$ such that $(T_\ell-\lambda_\ell)^{n_\ell} f =0$.} $f_i$}, but this fact does not seem to 
be of immediate use, for two reasons.  The methods for treating genuine eigenforms do not seem to apply readily to generalized eigenforms, and moreover it is not clear how to obtain an asymptotic formula for ${\pi}(f,x)$ from asymptotics for ${\pi}(f_i,x)$.   

For $f$ an eigenform, the main tool in Serre's study is the  Galois representation over a finite field  attached to $f$ by Deligne's construction,  $\rhob_f: G_{\Q,Np} \rightarrow \Gl_2(\F)$.  To deal with a general modular form $f$ we replace $\rho_f$ by a two dimensional Galois {\sl pseudo-representation}, $t_f$,  of $G_{\Q,Np}$ over a {\sl finite ring} $A_f$.  The ring $A_f$ is obtained as the quotient of $A$ by the annihilator of $f$, where $A$ is the Hecke algebra acting on the space of modular forms $M(N,\F)$.  The ring $A_f$ is not in general a field. In fact, it is a field precisely when $f$ is an eigenform for the Hecke operators $T_\ell$ ($\ell \nmid Np$).  
The Hecke algebra $A$ (at least in the case of $\Gamma_0(N)$) was introduced and studied in the wake of Swinnerton-Dyer's work on congruences between modular forms by Serre, Tate, Mazur, Jochnowitz and others.  More recent progress on understanding its structure may be  found in \cite{NS, NS2, BK}.  
In section~\ref{heckealgebra}, we recall the definitions of the Hecke algebra $A$, its quotient $A_f$, the pseudo-representation $t_f$, and collect together the results we need about them.

To prove Theorem 1, we introduce the notion of a {\sl pure form}.  A form $f$ is {\sl pure} if every Hecke operator $T_\ell$ (with $\ell \nmid Np$) in $A_f$ is either invertible or nilpotent.   Generalized eigenforms are pure since the finite ring $A_f$ is local in this case, 
but there are pure forms that are not generalized eigenforms.  For pure forms we can give a reasonable description of the set of integers $n$ with $(n,Np)=1$ and such that $a_n\neq 0$, and using this and a refinement of the Selberg-Delange method (see \S \ref{selberg}) we deduce (in \S \ref{proofasymppure}) an asymptotic formula for the number 
of $n\le x$ with $a_n\neq 0$ and $(n,Np)=1$.    For a general $f$, we show in \S \ref{proofasymp2} that if $f =\sum_i f_i$ is a minimal decomposition of $f$ into pure forms, then ${\pi}(f,x)$ is asymptotically $\sum_i {\pi}(f_i,x)$.  To complete the proof of Theorem 1, it 
remains to handle coefficients $a_n$ with $(n,Np)>1$, and this is treated in \S \ref{pfasympgeneralcase}.

\par \bigskip
Theorem~\ref{main} gives an asymptotic for the number of $n<x$ such that $a_n \neq 0$, but says nothing about the number of $n<x$ such that $a_n = a$, where $a$ is a specific fixed value in $\F^\ast$.  Some partial results are given during the course of the proof of Theorem~\ref{main} in \S\ref{proofasymppure}.  We say that $f$ has the {\it equidistribution property} if   
 the number of $n<x$ such that $a_n =a$ is asymptotically the same for every $a \in \F^\ast$. In \S\ref{equidis} we give sufficient conditions, and in some cases, necessary conditions, for the equidistribution property.
\par \bigskip
In section \S\ref{sectionvariant} we consider a variant of the main theorem, where one counts only the non-zero coefficients at square-free integers of a modular form. 
\par \bigskip
Let us finally mention that the constants $\alpha(f)$, $h(f)$ and $c(f)$ of Theorem~\ref{main} can be effectively computed from our proof. This is done in some cases in the last section, \S\ref{sectionexamples}.  However, we do not have a satisfactory understanding of how $h(f)$ and $c(f)$ behave as $f$ varies.  
Such an understanding would require a more detailed study of the structure of the Hecke algebra $A$ and of the space $M(N,\F)$ as a Hecke-module than is currently available (except in the case $p=2$, $N=1$: see \cite{NS2} and \cite{BN}; and partially in the case $p=3$, $N=1$: see \cite{med}).

{\bf Acknowledgments.}  We are grateful to the referee for a very careful reading of the paper.

\section{Applications of the Landau-Selberg-Delange method}

\label{selberg}

\subsection{Frobenian and multi-frobenian sets}

If $\Sigma$ is a finite set of primes, and $L$ a finite Galois extension of $\Q$ unramified outside $\Sigma$ and $\infty$, then for any prime $\ell \not \in \Sigma$,
we denote by  $\Frob_\ell \in \Gal(L/\Q)$ an element of Frobenius attached to $\ell$. We recall that $\Frob_\ell$ is only well-defined up to conjugation in
$\Gal(L/\Q)$.

\begin{definition} \label{multiform} Let $h$ be a non-negative integer and $\Sigma$ a finite set of primes.  We say that a set $\Mc$ of positive integers is {\it $\Sigma$-multi-frobenian} of height $h$ if there exists a finite Galois extension $L$ of $\Q$ with Galois group $G$, unramified outside $\Sigma$ and infinity, and a subset $D$ of $G^h$ invariant under conjugation and under permutations of the coordinates, such that
$m \in \Mc$ if and only if $m=\ell_1 \dots \ell_h$ where the $\ell_i$ are distinct primes not in $\Sigma$, 
and $(\Frob_{\ell_1}, \dots, \Frob_{\ell_h}) \in D$.  For such a $\Sigma$-multi-frobenian set $\Mc$ we define its density $\delta(\Mc)$ to be 
$$
\delta(\Mc)= \frac{\# D}{h! (\# G)^h}.
$$
\end{definition} 
\par \medskip
Observe that the condition $(\Frob_{\ell_1},\dots, \Frob_{\ell_h}) \in D$ depends only on the product $\ell_1 \dots \ell_h$, since replacing each
$\Frob_{\ell_i}$ by a conjugate in $G$ amounts to replacing $(\Frob_{\ell_1},\dots, \Frob_{\ell_h})$ by a conjugate in $G^h$ and $D$ is invariant by conjugacy in $G^h$, and since changing the order of the prime factors $\ell_1,\dots,\ell_h$ permutes the components of  $(\Frob_{\ell_1},\dots, \Frob_{\ell_h})$ and $D$ is invariant by permutations.  Thus the notion of a multi-frobenian set is well defined.

There is only one $\Sigma$-multi-frobenian set of height $h=0$, namely $\{1\}$.
A $\Sigma$-multi-frobenian set of height $1$ is just a $\Sigma$-frobenian set of prime numbers in the usual sense, see \cite[\S3.3.1]{SerreNXp}.  In what follows we will say that a set is {\it multi-frobenian} if it is $\Sigma$-multi-frobenian 
for some finite set of primes $\Sigma$, and {\it frobenian} if it is multi-frobenian of height 1. We observe that this definition of frobenian is slightly 
more restrictive that the one used by Serre (cf. \cite[\S3.3.2]{SerreNXp}) for whom a set of primes is frobenian if it is frobenian in our sense up to a finite set of primes.
The more restrictive definition of frobenian that we adopt here will be sufficient for our purposes, and we hope that its use will 
cause no confusion to the reader.

 \begin{lemma} \label{lemmamultifrob} Let ${\mathcal M}$ be a multi-frobenian set of height $h$ and density $\delta({\mathcal M})$.  Then 
$$ 
\sum_{\substack{{m\in {\mathcal M}} \\ {m\le x}}} \frac{1}{m} \sim \delta({\mathcal M}) (\log \log x)^{h} .
$$  
\end{lemma} 
\begin{pf}  This follows from the Chebotarev density theorem. 
\end{pf} 

Note in particular that $\delta(\Mc)$ depends only on the set $\Mc$ and not 
on the choice of $L$, $G$ and $D$.

\begin{remark} Using the Chebotarev density theorem, one may show that if $\Mc$ is a multi-frobenian set of height $h$, then
$$
| \{ n\le x: \ n\in {\mathcal M}\} |  \sim h \delta({\mathcal M}) \frac{x}{\log x}(\log \log x)^{h-1}.
$$
This formula clearly implies Lemma 3 by partial summation, but the weaker Mertens-type estimate of  Lemma 3 suffices for our purposes. 
\end{remark}  

\subsection{Square-free integers with prime factors in a frobenian set and random walks}

We begin with a general result of the Landau-Selberg-Delange type, which follows by the method discussed in 
Chapter II.5 of Tenenbaum's book \cite{Ten}, or as in Th\'eor\`eme 2.8 of Serre's paper \cite{SerreE}.   

\begin{prop}   \label{propsqfree} Let $a(n)$ be a sequence of complex numbers with $|a(n)|\le d_k(n)$ for some 
natural number $k$, where $d_k(n)$ denotes the $k$-divisor function defined by $\zeta(s)^k = \sum_{n=1}^{\infty} d_k(n)n^{-s}$.  Suppose that in the region Re$(s)>1$ the function $A(s)=\sum_{n=1}^{\infty} a(n)n^{-s} $ 
can be written as (for some real number $\alpha$)
$$ 
A(s) = \zeta(s)^{\alpha} B(s), 
$$ 
where $B(s)$ extends analytically to the region Re$(s)> 1- c/\log (2+|t|)$ for some positive constant $c$, and 
is bounded in that  region by $|B(s)| \le C(1+|t|)$ for some constant $C$.   Then for all $x\ge 3$ and any $J\ge 0$, 
there is an asymptotic expansion  
$$ 
\sum_{n\le x} a(n) = \sum_{j=0}^{J} \frac{A_j x}{(\log x)^{1+j-\alpha}} + O \Big( \frac{C x}{(\log x)^{J+2-\alpha}}\Big),  
$$ 
where the $A_j$ are constants, with 
$$
A_0= \frac{B(1)}{\Gamma(\alpha)},
$$ 
and the implied constant in the remainder term depends only on $c$, $k$, and  $J$.
\end{prop} 
\begin{pf} As mentioned above, this is a straightforward application of the Landau-Selberg-Delange method, and so 
we content ourselves with sketching the argument briefly.  The constant $c$ can be replaced by a possibly smaller constant so that $\zeta(s)$ has no zeros in the 
region Re$(s) >1-c/\log (2+|t|)$, and moreover in this region we have the classical bounds $|\zeta(s)^{\alpha}| \ll (\log (|s|+2))^{A|\alpha|}$ for some constant $A$ provided we stay away from $s=1$ (see for example II.3 of \cite{Ten}).    
 Next, by applying a quantitative version of Perron's formula we 
see that for $x\ge 3$ and with $x^{\frac{1}{10k}} \ge T \ge 1 $, 
$$ 
\sum_{n\le x} a(n) = \frac{1}{2\pi i} \int_{1+1/\log x -iT}^{1+1/\log x +iT} A(s) \frac{x^s}{s} ds + O\Big( \frac{x}{T} (\log x)^{k} \Big). 
$$ 
Now we deform the line of integration as follows.  First make a slit along the real line segment from 
$1-c/\log (T+2)$ to $1$.  Then from $1+1/\log x+iT$ we proceed in a straight line to $1-c/\log (T+2)+iT$ and 
from there to $1-c/\log (T+2)+ i 0^+$ (on the upper part of the slit) and proceed from there to $1$ and then circle around 
to the lower part of the slit until $1-c/\log (T+2) +i 0^-$ and from there to $1-c/\log (T+2) -iT$ and thence to $1+1/\log x -iT$.  
The integrand has a logarithmic singularity at $1$, and the change in the argument above and below the slit 
leads to the main terms in the asymptotic expansion (by ``Hankel's formula" see \S II.5.2 of Tenenbaum \cite{Ten}).  The 
remaining integrals are estimated using the bounds for $|\zeta(s)^{\alpha}|$ in the zero-free region, together with 
our assumed bound for $|B(s)|$.  The resulting error terms are bounded by $O(x^{1-c/\log (T+2)} (T+2) \log (T+2))$.  Choosing 
$T =\exp(c_1 \sqrt{\log x})$ for a suitably small positive constant $c_1$, we obtain the proposition. 
\end{pf}
   
 Now suppose we are given a frobenian set of primes $\Uc$ of density $\beta =\delta(\Uc)> 0$,  a finite abelian group $\Gamma$, and a frobenian map\footnote{A map from a frobenian set of primes to a finite set is said {\it frobenian} if its fibers are frobenian.}
 $\tau_0: {\Uc} \to \Gamma$ such that the image $\tau_0(\Uc)$ generates $\Gamma$.
  Using multiplicativity, extend $\tau_0$ to  a map $\tau$ from the set of square-free numbers composed of prime factors in ${\mathcal U}$ to $\Gamma$.  
 
 \begin{theorem} \label{thm5}  Let $g$ be any given element of $\Gamma$, and let $r$ be a positive integer. Then, for $x\ge 3$ and uniformly in $r$, we have 
 $$ 
 \#\{ n\le x: n \text{ square-free}, \ p \mid n \implies p\in {\Uc}, \ \tau(n)=g, \ (n,r)=1\}
 $$  
 equals 
 $$ 
 C({\mathcal U}, r)\frac{1}{|\Gamma|} \frac{x}{(\log x)^{1-\beta} }  
 + O\Big( \frac{x d(r)}{(\log x)^{1-\beta+\delta}}\Big),  
 $$ 
 where ${ \mathcal C}({\Uc},r) = \frac{1}{\Gamma(\beta)} \prod_{p} w_p$ with 
 $w_p=(1+1/p) (1-1/p)^{\beta}$ if $p\in {\Uc}$ with $p\nmid r$, and $w_p= (1-1/p)^{\beta}$ otherwise.  
In the remainder term above, $d(r)$ denotes the number of divisors of $r$, and $\delta$ is a fixed positive number (depending only on the 
 group $\Gamma$).  
 \end{theorem} 
 \begin{pf}  We use the orthogonality of the characters of the group $\Gamma$, which we 
 write  multiplicatively even though it is abelian.  Thus the 
quantity we want is 
$$ 
\frac{1}{|\Gamma|} \sum_{\chi \in {\widehat \Gamma}} \overline{\chi(g)} \sum_{\substack{{n\le x} \\ {(n,r)=1}}} \chi(\tau(n)),
$$ 
where we set  $\chi(\tau(n))=0$ if $n$ is divisible by some prime not in ${\mathcal U}$ or if $n$ is not square-free.    

We will use Proposition~\ref{propsqfree} to evaluate the sum over $n$ above.  Since the map $\tau$ is frobenian, by the 
usual proof of the Chebotarev density theorem (that is, by expressing frobenian sets in terms of Hecke $L$-functions, and using the zero-free region for Hecke $L$-functions) we may write 
$$ 
\sum_{\substack{ n=1\\ (n,r)=1}}^{\infty} \frac{\chi(\tau(n))}{n^s} = \zeta(s)^{\beta(\chi)} B_{\chi,r}(s), 
$$ 
where 
$$ 
\beta(\chi) = \sum_{g \in \Gamma} \chi(g) \delta(\tau_0^{-1}(g)), 
$$ 
and $B_{\chi,r}(s)$ extends analytically to the region Re$(s)>1-c/(\log (2+|t|))$ for some $1/10\ge c>0$, and in that region 
satisfies the bound $|B_{\chi,r}(s)| \le C d(r) (1+|t|)$ for some constant $C$.  The constants $c$ and $C$ depend only on ${\Uc}$ and ${\Gamma}$ but not on $r$.

First suppose that $\chi$ equals the trivial character $\chi_0$.  Note that $\beta(\chi)$ then equals $\beta$, and 
$$ 
B_{\chi_0,r}(s) = \prod_{\substack{{p\in {\mathcal U}} \\ {p\nmid r}}} \Big(1-\frac{1}{p^s}\Big)^{\beta}\Big(1+\frac 1{p^s}\Big) 
\prod_{\substack{{p \notin {\mathcal U}} \\ {\text{ or } p| r}}} \Big(1-\frac 1{p^s}\Big)^{\beta}. 
$$ 
Therefore, appealing to Proposition~\ref{propsqfree}, we obtain the main term of the theorem. 

Now suppose that $\chi$ is not the trivial character.  Then Re$(\beta(\chi)) \le \beta -\delta$ for some fixed $\delta>0$, since there is a  $g$ in the image of $\tau_0$
such that $\chi(g) \neq 1$ (since $\tau(\Uc)$ generates $\Gamma$), and the frobenian set $\tau_0^{-1}(g)$ is non-empty and hence of positive density $\delta(\tau_0^{-1}(g))$.
 Therefore, by Proposition~\ref{propsqfree}, we see that the contribution of  the non-trivial characters is 
$$ 
O\Big( \frac{x d(r)}{(\log x)^{1-\beta+\delta}}\Big), 
$$ 
completing the proof of the theorem.
\end{pf}

\subsection{A density result}

We keep the notations and hypotheses of the preceding section: $\Uc$ is a frobenian set with $\beta=\delta(\Uc)>0$, $\Gamma$ is a finite abelian group,  and $\tau_0: \Uc \rightarrow \Gamma$ is a frobenian map whose image generates $\Gamma$. In addition, let ${\mathcal M}$ be a multi-frobenian set of height $h\ge 0$, 
such that every element in ${\mathcal M}$ is coprime  to the primes in ${\mathcal U}$.  Let ${\mathcal S}$ be a given non-empty set of square-full numbers (we permit $1$ to be treated as a square-full number). 

Define $\Zc=\Zc(\Uc,\Mc,\Sc)$ to be the set of positive integers $n \geq 1$ that can be written as 
\begin{num} \label{Zcn}
$n = m m' m''$
with $m,m',m''$ positive integers such that 
\begin{sousnum}\label{Zcn1}
 $m$ is square-free and all its prime factors are in $\Uc$,
 \end{sousnum}
 \begin{sousnum}\label{Zcn2}
 $m' \in \Mc$,
 \end{sousnum}
 \begin{sousnum} \label{Zcn3}
 $m'' \in \Sc$ and $m''$ is relatively prime to $m m'$.
\end{sousnum} \end{num}
These conditions imply that $m$, $m'$ and $m''$ are pairwise relatively prime, and for $n \in \Zc$ such a decomposition $n=mm'm''$ is unique.
Extend $\tau$ to a map $ \Zc \rightarrow \Gamma$ by setting $\tau(n)=\tau(m)$ for $n$ as in \ref{Zcn}.
Let $\Delta$ be any non-empty subset of $\Gamma$.

\begin{theorem} \label{delangemethod} 
With notations as above, we have 
$$ 
\# \{ n\le x: \ n\in {\mathcal Z}, \ \tau(n) \in \Delta \}  \sim C  \delta({\mathcal M}) \frac{|\Delta|}{|\Gamma|} \frac{x}{(\log x)^{1-\beta} }(\log \log x)^{h},  
$$ 
where (with $C(\Uc,s)$ as in Theorem~\ref{thm5})

$$ 
C = \sum_{s\in {\mathcal S}} \frac{C({\mathcal U},s)}{s}.
$$ 
 \end{theorem}  
\begin{pf} Put $R= (\log x)^2$, and $z= x^{1/\log \log x}$.   We want to count $n= m m^{\prime} m^{\prime \prime}$ 
with $m^{\prime \prime} \in {\mathcal S}$, $m^{\prime} \in {\mathcal M}$ with $(m^{\prime},m^{\prime\prime})=1$, and 
$m$ composed of primes in ${\mathcal U}$ with $(m,m^{\prime\prime})=1$ and $\tau(m)=g$.   We now group these 
terms according to whether (i) $m^{\prime \prime} \le R$ and $m^{\prime} \le z$, or (ii) $m^{\prime \prime} \le R$ but $m^{\prime}>z$, or 
(iii) $m^{\prime \prime} >R$.  We shall show that the first case gives the main term in the asymptotics, and the other two cases 
are negligible.  

First consider case (i).   This case contributes 
$$ 
\sum_{\substack{{m^{\prime \prime} \in {\mathcal S}} \\ {m^{\prime \prime} \le R} } }  
\sum_{\substack{{m^{\prime} \in {\mathcal M} } \\ {(m^{\prime},m^{\prime \prime})=1} \\ 
{m^{\prime} \le z}} } \sum_{g\in \Delta} \Big| \Big\{ m\le \frac{x}{m^{\prime} m^{\prime\prime}}: \ \tau(m) =g, \ (m,m^{\prime\prime})=1 \Big\}\Big|.
$$ 
Now we use Theorem \ref{thm5}, so that the above equals 
$$ 
\sum_{\substack {{m^{\prime \prime} \in {\mathcal S}} \\ {m^{\prime \prime} \le R} } }\sum_{\substack{{m^{\prime } \in {\mathcal M} } \\ {(m^{\prime},m^{\prime \prime})=1} \\ 
{m^{\prime} \le z}} } \Big( C({\mathcal U}, m^{\prime\prime}) \frac{|\Delta|}{|\Gamma| } \frac{x}{m^{\prime}m^{\prime \prime} (\log (x/m^{\prime} m^{\prime\prime}))^{1-\beta}} + O\Big( \frac{x d(m^{\prime\prime})}{m^{\prime}m^{\prime\prime}(\log x)^{1-\beta+\delta}}\Big)\Big). 
$$ 
Using Lemma~\ref{lemmamultifrob}, and since $\sum_{m^{\prime \prime} \in {\mathcal S} } d(m^{\prime\prime})/m^{\prime \prime}$ converges, we see that the error term above is $O(x/(\log x)^{1-\beta +\delta -\epsilon})$, which is negligible.  
Since $\log (x/m^{\prime} m^{\prime \prime}) \sim \log x$, the main term above is (again using Lemma~\ref{lemmamultifrob}) 
$$ 
\sim\frac{|\Delta|}{|\Gamma|} \frac{x}{(\log x)^{1-\beta}} \Big( \delta({\mathcal M}) (\log \log x)^h\Big) \sum_{\substack{ 
{m^{\prime \prime } \in {\mathcal S}} \\ {m^{\prime \prime} \le R} } } \frac{C({\mathcal U},m^{\prime \prime})}{m^{\prime \prime}},
$$
which equals the main term of the theorem.  
 
Now consider case (ii).  Since all the terms involved are positive, we see that these terms contribute (with $\omega(u)$ denoting the number of distinct prime factors of $u$)
\begin{equation}
\label{eqn2.1}
\ll \sum_{\substack{ {m^{\prime \prime} \in {\mathcal S}} \\ {m^{\prime\prime}\le R}} }\sum_{\substack{ 
{z\le u \le x/m^{\prime\prime}} \\ {\omega(u)=h}}} \sum_{\substack{ {m\le x/(um^{\prime \prime}) }\\ 
{p|m \implies p\in {\mathcal U}}} } 1. 
\end{equation} 
Now in the sums above either $u\le \sqrt{x}$, or $m\le \sqrt{x}$.  In the first case, note that the largest prime factor of $u$ lies in $[z^{1/h},\sqrt{x}]$ and 
the others are all below $\sqrt{x}$.  Moreover, using Proposition \ref{propsqfree}, the inner sum over $m$ in \eqref{eqn2.1} is 
$\ll x/(um^{\prime \prime} (\log x)^{1-\beta})$.   Thus we see that the first case contribution to \eqref{eqn2.1} is bounded by  
\begin{align*}
&\ll \sum_{\substack{{m^{\prime\prime} \in {\mathcal S}} \\ {m^{\prime \prime} \le R}}} 
\sum_{\substack{z< u \le \sqrt{x} \\ \omega(u) = h}} \frac{x}{u m^{\prime\prime} (\log x)^{1-\beta}} 
\ll \frac{x}{(\log x)^{1-\beta}} \Big(\sum_{z^{1/h} \le p \le \sqrt{x}} \frac{1}{p} \Big) \Big(\sum_{p\le \sqrt{x}} \frac 1p \Big)^{h-1} \\
&\ll \frac{x}{(\log x)^{1-\beta}} (\log \log x)^{h-1} \log \log \log x.
\end{align*}
For the second case, note that for $m\le \sqrt{x}$ (and $m^{\prime\prime} \le R=(\log x)^2$) we have (by standard estimates for 
the number of integers with $h$ distinct prime factors)
$$ 
\sum_{\substack{ u\le x/(mm^{''}) \\ \omega(u)=h}} 1 \ll \frac{x}{mm^{''} } \frac{(\log \log x)^{h-1}}{\log x}, 
$$ 
and so we obtain that the second case contribution to \eqref{eqn2.1} is bounded by 
\begin{align*}
&\ll \frac{x}{\log x} (\log \log x)^{h-1} \sum_{\substack { m\le \sqrt{x} \\ m\in {\mathcal U}} } \frac{1}{m} \ll 
\frac{x}{\log x} (\log \log x)^{h-1} \prod_{\substack{ p\le \sqrt{x}\\  p\in {\mathcal U}} } \Big( 1+\frac{1}{p} \Big) 
\\
&\ll \frac{x}{(\log x)^{1-\beta}} (\log \log x)^{h-1}.  
\end{align*}  
Putting both cases together, we conclude that the contribution of the terms in case (ii) 
is 
$$ 
\ll \frac{x}{(\log x)^{1-\beta}} (\log \log x)^{h-1} \log \log \log x, 
$$ 
which is small compared to the contribution from case (i).  

Finally, since the number of $m m^{\prime} \le x/m^{\prime\prime}$ is trivially at most $x/m^{\prime \prime}$, the contribution in case (iii) is 
$$
\ll \sum_{\substack{{m^{\prime\prime } \in {\mathcal S}} \\ {m^{\prime \prime}>R}} }\frac{x}{m^{\prime \prime}} \ll 
\frac{x}{\sqrt{R}} = \frac{x}{\log x}, 
$$ 
which is negligible.  This completes our proof.  
\end{pf}

\section{Modular forms modulo $p$}

\label{heckealgebra}
\subsection{The algebra of modular forms $M(N,\F)$} 

As in the introduction, we fix an odd prime $p$ and a level $N \geq 1$.   Let $k\ge 0$ be an integer.  The space $M_k(N,{\Z})$ 
denotes the space of all holomorphic modular forms of weight $k$, level $\Gamma_1(N)$, and with $q$-expansion at infinity in ${\Z}[[q]]$.  For any commutative ring $A$ we define 
$$
M_k (N,A) = M_k(N,\Z) \otimes A.
$$
 The natural $q$-expansion map $M_k(N,A) \rightarrow A[[q]]$ is injective for any ring $A$ (this is the $q$-expansion principle, cf. \cite[Theorem 12.3.4]{DI}), and so we may view below $M_k(N,A)$ as a subspace of $A[[q]]$.  Finally we define  
 $$
 M(N,A) = \sum_{k=0}^\infty M_k(N,A) \subset A[[q]].
 $$

Note that if $A$ is a subring of $\C$, then $M(N,A)$ is the {\sl direct} sum of the spaces $M_k(N,A)$ (see \cite[Lemma 2.1.1]{miyake}).  However the situation is 
different for general rings $A$, and in particular when $A$ is a finite field.  For instance, the constant modular form $1$
of weight $0$ in $M_0(N,\F_p)$ and the Eisenstein series $E_{p-1}$ in $M_{p-1}(N,\F_p)$ both have the same $q$-expansion $1$, showing that the subspaces $M_0(N,\F_p)$ and $M_{p-1}(N,\F_p)$ are not in direct sum in $\F_p[[q]]$. For the same reason it is not true that $M(N,A) \otimes_A A' = M(N,A')$ in general (though this is  true if $A'$ is flat over $A$); rather $M(N,A')$ is the image of $M(N,A) \otimes_A A'$ in $A'[[q]]$.

\subsection{Hecke operators on $M_k(N,A)$} 
\label{subHecke}

For any $k \geq 0$, the space of modular forms $M_k(N,\C)$ is endowed with the action of the Hecke operators $T_n$ for positive integers $n$.   If $n$ is a positive integer coprime to $N$, define the operator $S_n$ as  $n^{k-2}\langle n \rangle$, where $\langle n \rangle$ is the diamond operator.  
Recall that these operators satisfy the following properties. 

\begin{num} \label{heckecommute} All the operators $T_n$ and $S_m$ commute. \end{num}
\begin{num} \label{Snmult} We have $S_1=1$ and $S_{mn}=S_m S_{n}$ for all $m, n$ coprime to $N$.\end{num}
\begin{num} \label{Tnmult} The Hecke relations $T_1=1$, $T_{mn}=T_mT_n$ if $(m,n)=1$, hold.  If $\ell \nmid N$ is a prime, then  $T_{\ell^{n+1}} = T_{\ell^n} T_{\ell} - \ell S_{\ell} T_{\ell^{n-1}}$.  If $\ell |N$ is prime then  $T_{\ell^n}=(T_\ell)^n$.
\end{num}
As is customary, we shall also use below the notation $U_{\ell}$ for the operators $T_{\ell}$ when $\ell \mid N$. 
From the above relations one sees that the operators $T_\ell$ and $S_\ell$ for $\ell$ prime determine all the others. 
Recall that the action of the Hecke operators on $q$-expansions is given as follows. 

\begin{num} \label{qexp1} If $\ell |N$ then $a_n(U_\ell f) = a_{\ell n} (f)$. \end{num}
\begin{num} \label{qexp2} If $\ell \nmid N$ is prime, then $a_n(T_\ell f) = a_{\ell n}(f) + \ell a_{n/\ell}(S_\ell f)$, with the 
understanding that $a_{n/\ell}$ means $0$ if $\ell \nmid n$.\end{num}
\noindent It follows that:
\begin{num} \label{qexp3} If $(n,m)=1$ then $a_n(T_m f) = a_{nm}(f)$. In particular $a_1(T_m f)=a_m(f)$ for every $m \geq 1$.
\end{num}

Lastly, we recall the following important fact, which follows from the geometric interpretation due to Katz (\cite{katz}) of the elements of $M_k(N,A)$ as the sections of a coherent sheaf on the
modular curve $Y_1(N)_{/A}$ over $A$, and of the Hecke operators as correspondences on $Y_1(N)$. A convenient reference is \cite[Chapter 12]{DI}.

\begin{num} \label{heckeonMk} Let $A$ be a subring of $\C$.   All the operators $T_n$ and $S_n$ leave stable the subspace $M_k(N,A)$ of $M_k(N,\C)$.
\end{num}
 
This fact allows us to define unambiguously the operators $T_n$ and $S_n$ over $M_k(N,A)=M_k(N,\Z) \otimes_\Z A$ by extending the scalars from $\Z$ to $A$ for the linear operators $T_n$ and $S_n$ on $M_k(N,\Z)$.

 \subsection{Hecke operators on $M(N,\F)$}

\label{secHeckeonM} 
From now on, $\F$ is a finite field of characteristic $p$.  First we recall a result due to Serre and Katz, which allows us to assume that the level $N$ is prime to $p$;  for a proof, see \cite[pages 21-22]{gouvea}.

\begin{num} \label{levelred} Let $\F$ be a finite field of characteristic $p$. Write $N=N_0p^\nu$ with $(N_0,p)=1$. Then as subspaces of $\F[[q]]$ one has $M(N,\F_p)=M(N_0,\F_p)$. \end{num}

Henceforth, we assume that $(N,p)=1$.  

\begin{num}  \label{heckeonM} There are unique operators $T_n$ (for any $n \geq 1$) and $S_n$ (for $n \geq 1$ with $(n,N)=1$) on $M(N,\F)$ such that, for any $k \geq 0$, the inclusion $M_k(N,\F) \hookrightarrow M(N,\F)$ is compatible with the operators $T_n$ and $S_n$ defined on the source and target.
\end{num}

Since the sum of the $M_k(N,A)$ for $k=0,1,2,\dots$ is $M(N,A)$ by definition, the uniqueness claimed in \ref{heckeonM} follows. The existence relies on the interpretation of the 
elements of $M(N,A)$ as algebraic functions on the open Igusa curve (an étale cover of degree $p-1$ of the ordinary locus of $Y_1(N)_{/\F_p}$) which is due to Katz (see \cite{katz}, \cite[Theorem 2.2]{katzCD}) based on earlier work of Igusa. 
For a more recent reference for \ref{heckeonM}, see \cite[Propositions 5.5 and 5.9]{gross}.
  
It is clear that the operators $T_n$ and $S_n$ still satisfy properties \ref{heckecommute} to \ref{qexp3}.
We record one more easy consequence of \ref{heckeonM}.

\begin{num} \label{locfin} The actions of the Hecke operators $T_n$ and $S_n$ on $M(N,\F)$ are locally finite. That is, 
any form $f \in M(N,\F)$ is contained in a finite-dimensional subspace of $M(N,\F)$ stable under all these operators.
\end{num}

We shall use the notation $U_p$ instead of $T_p$ when acting on the space $M(N,\F)$. More generally, if $m$ is an integer all of whose prime factors divide $Np$ we shall use the notation $U_m$ instead of $T_m$.

Finally, we note that the space $M(N,\F)$ enjoys an additional Hecke operator, see \cite[\S1]{jochnowitz2}.

\begin{num} \label{Vp} The subspace $M(N,\F)$ of $\F[[q]]$ is stable under the operator $V_p$, defined by $V_p(\sum a_n q^n) = \sum a_n q^{pn}$.
\end{num}

\subsection{The subspace $\Fc(N,\F)$ of $M(N,\F)$}

Using the same notation as in \cite{NS}, \cite{NS2}, let us define $\Fc(N,\F)$ as the subspace $\cap_{\ell \mid Np} \ker U_\ell$ of $M(N,\F)$.
In other words \begin{num} \label{defMp}$ \Fc(N,\F)=\{f = \sum_{n=0}^\infty a_n q^n \in M(N,\F),\ a_n \neq 0 \Rightarrow (n,Np)=1 \}$.\end{num}
Since the Hecke operators commute, the operators $T_\ell$ and $S_\ell$ for $\ell \nmid Np$ stabilize $\Fc(N,\F)$.

\subsection{The residual Galois representations $\rhob$ and the invariant $\alpha(\rhob)$}
\label{residual}
We denote by $G_{\Q,Np}$ the Galois group of the maximal algebraic extension of $\Q$ unramified outside $Np$. We denote by $c$ a complex conjugation in $G_{\Q,Np}$.
If $\ell$ is a prime not dividing $Np$, we denote by $\Frob_{\ell}$ an element of Frobenius associated to $\ell$ in $G_{\Q,Np}$. We fix an algebraic closure $\bar \F_p$ of $\F_p$.

We shall denote by $R=R(N,p)$ the set of equivalence classes of continuous odd\footnote{That is, such that $\tr \rhob(c)=0$.} semi-simple two-dimensional representations $\rhob$ of the Galois group $G_{\Q,Np}$ over $\bar \F_p$ that are attached to eigenforms in $M(N,\bar \F_p)$. Here we say that $\rhob$ is 
{\it attached to a an eigenform} in $M(N,\bar \F_p)$ if there exists a non-zero eigenform $f \in M(N,\bar \F_p)$ for the Hecke operators $T_\ell$ and $S_\ell$ for $\ell \nmid Np$, with eigenvalues $\lambda_\ell$ and $\sigma_\ell$, such that 
\begin{num} \label{es} the characteristic polynomial of $\rhob(\Frob_\ell)$ is $X^2-\lambda_\ell X + \ell \sigma_\ell$. \end{num} 
Although we do not need this fact, we remark that Khare and Wintenberger have shown Serre's conjecture that every odd semi-simple two-dimensional representation of Serre's conductor $N$ is attached to an eigenform in $M(N,\bar \F_p)$.

A result of  Atkin, Serre and Tate in the case $N=1$ (\cite{SerreB}), and of Jochnowitz in the general case (\cite[Theorem 2.2]{jochnowitz2}) states that the number of systems of eigenvalues for the $T_\ell$ and $S_\ell$ appearing in $M(N,\bar \F_p)$ is finite. Hence $R(N,p)$ is a finite set. If $\rhob: G_{\Q,Np} \rightarrow \Gl_2(\bar \F_p)$ is a representation, it is defined over some finite extension $\F$ of $\F_p$ inside $\bar \F_p$ (for absolutely irreducible $\rhob$, this amounts to saying that $\tr \rhob(G_{\Q,Np}) \subset \F$, since finite fields have trivial Brauer groups). Therefore, there exists a finite extension $\F$ of $\F_p$ such that all representations in $R(N,p)$ are defined over $\F$. 
\par \bigskip
For $\rhob \in R(N,p)$, we shall denote by $U_\rhob$ the open and closed subset of $G_{\Q,Np}$ of elements $g$ such that $\tr \rhob(g) \neq 0$,
and by $N_\rhob$ its complement, the set of elements $g$ such that $\tr \rhob(g)=0$. We set $\alpha(\rhob) = \mu_{G_{\Q,Np}}(N_\rhob)$, where $\mu_{G_{\Q,Np}}$ is the Haar measure on the compact group $G_{\Q,Np}$. 


\begin{prop} \label{alpharho} For all representations $\rhob$ we have $\alpha(\rhob) \in \Q$ with $0<\alpha(\rhob) \le 3/4$.  If $\rhob$ is reducible, we have $\alpha(\rhob) \leq 1/2$.
\end{prop} 
\begin{pf} By definition, $\alpha(\rhob)$ is the proportion of elements of trace zero in the finite subgroup $G = \rhob(G_{\Q,Np})$ of $\Gl_2(\bar \F_p)$.  Thus $\alpha(\rhob)$ is rational and is at most one.  Since $\rhob(c)$ has trace zero, we have $\alpha(\rhob) >0$.  It remains now to obtain the upper bounds claimed for $\alpha(\rhob)$.  
Let $G'$ be the image of $G$ in $\PGL_2(\bar \F_p)$. Then $\alpha(\rhob)$ is also the proportion of elements of trace zero in $G'$ (it makes sense
to say that an element  of $\PGL_2(\bar \F_p)$ has ``trace zero", even though the trace of such an element is of course not well-defined). 
Also, observe that an element $g'$ in $\PGL_2(\bar \F_p)$ has trace $0$ if and only if it has order exactly $2$. Indeed, let $g$ be a lift of $g'$ in $\Gl_2(\bar \F_p)$. If $g$ is diagonalizable, and $x,y$ are its eigenvalues, then $g'$ has order exactly $2$ means that $x \neq y$, but $x^2 = y^2$; thus $x= -y$, and $\tr g=0$.  If $g$ is not diagonalizable, then the order of $g'$ is a power of $p$, hence not $2$, and it has a double eigenvalue $x \neq 0$ so its trace $2x$ is not $0$. 
 Hence $\alpha(\rhob)$ is also the proportion of elements of order $2$ in $G'$.

If $\rhob$ is reducible, then, since $\rhob$ is assumed semi-simple, $G$ is conjugate to a subgroup of the diagonal subgroup $D=\bar \F_p^\ast \times \bar \F_p^\ast$, and $G'$ may thus be assumed to be a subgroup of the image $D'$ of $D$ in $\PGL_2$. The group $D'$ is isomorphic to $\bar \F_p^\ast$, by the isomorphism sending $x \in \bar \F_p^\ast$ to the image of $\mat{1 & 0 \\ 0 & x}$ in $\PGL_2(\F_p)$, and via this identification, the only element of trace zero of $D'$ is $-1$, which is always in $G'$ because $G$ contains $\rhob(c)$. One thus  has $\alpha(\rhob) = 1 / |G'|$.
Therefore $\alpha(\rhob) \leq 1/2$ since $G'$ is not the trivial group because $\rhob(c)$ is not trivial in $\PGL_2(\bar \F_p)$.

Now assume that $\rhob$ is irreducible. We shall use the classification of subgroups of $\PGL_2(\bar \F_p)$ for which a convenient modern reference is \cite{faber}. According to Theorems B and C of \cite{faber}, if $G'$ is any finite subgroup of $\PGL_2(\bar \F_p)$,
 we are in one of 9 situations described there, and labeled B(1) to B(4) and C(1) to C(5). The case B(3) does not arise since we assume $p>2$, and neither do cases B(2) and C(1) which contradict the assumed irreducibility of 
 $\rhob$ (for B(2) because $G'$ cyclic implies $G$ abelian, and for C(1)  by Remark 2.1 of \cite{faber}).
  In the other situations, we argue as follows.
  
\begin{itemize}
\item[C(2)] $G'$ is isomorphic to a dihedral group $D_{2n}$ of order $2n$ for $n \geq 2$ an integer, which is a semi-direct product of a cyclic group $C_n$ by a subgroup of order $2$.
In this case, the elements of order $2$ are the elements not in $C_n$ and, if $n$ is even, the unique element of order $2$ in $C_n$. Thus
$$\alpha(\rhob)=\begin{cases} \frac{1}{2} & \text{ if  $n$ is odd} \\
\frac{1}{2}+\frac{1}{2n} & \text{ if  $n$ is even} \end{cases}
$$
Note that if $n=2$, $\alpha(\rhob)=3/4$, and in all other cases $\alpha(\rhob) \leq 5/8$.
 
\item[C(3)]$ G' \simeq A_4$, so 
$\alpha(\rhob)=\frac{1}{4}$ since $A_4$ has order 12, and has 3 elements of order 2.

\item[C(4)]   $G'   \simeq S_4$, so  $\alpha(\rhob)=\frac{3}{8}$ since $S_4$ has order 24 and 9 elements of order $2$ (6 transpositions and 3 products of two disjoint transpositions).

\item[C(5), B(4)] 
$G' \simeq A_5$, so $\alpha(\rhob) =  \frac{1}{4}$ since $A_5$ has order 60 and has 15 elements of order $2$ (the products of two disjoint transpositions).  

\item[B(1)]  The subgroup $G'$ of $\PGL_2(\bar \F_p)$ is conjugate to $\PGL_2(\F_q)$, where $q$ is some power of $p$. In this case, the number of matrices
 of trace $0$ in $G'$ is $q^2$, while $|G'|=q(q-1)(q+1)$, so $$\alpha(\rhob)=\frac{q}{(q-1)(q+1)}.$$
Thus in this case, we have $\alpha(\rhob) \leq 3/8$, and this bound is attained for $q=3$.

\item[B(1) again] The subgroup $G'$ of $\PGL_2(\bar \F_p)$ is conjugate to $\PSL_2(\F_q)$.
The number of matrices of trace $0$ in $\SL_2(\F_q)$ is $q^2-q$ if $-1$ is not a square in $\F_q$, and $q^2+q$ if $-1$ is a square. Since $|\SL_2(\F_q)|=q(q-1)(q+1)$ one has 
$$\alpha(\rhob)=\begin{cases} \frac{1}{q+1} & \text{ if $-1$ is not a square in $\F_q$} \\
\frac{1}{q-1} &  \text{ if $-1$ is a square in $\F_q$} \end{cases}
$$ Thus in this case, we have $\alpha(\rhob) \leq 1/4$, and this value is attained for $q=3$ and $q=5$.
\end{itemize}
 \end{pf}

\subsection{The Hecke algebra $A$} 

From now on, we assume that $\F$ is a finite field contained in $\bar \F_p$ and  large enough to contain the fields of definition of all the representations $\rhob \in R(N,p)$.

Let $A=A(N,\F)$ be the closed sub-algebra of $\End_\F(M(N,\F))$ generated by the Hecke operators $T_\ell$ and $S_\ell$ for $\ell$ prime not dividing $Np$.  Equivalently, by \ref{Tnmult}, $A$ is the closed sub-algebra of $\End_\F(M(N,\F))$ generated by the $T_m$ for all $m$ relatively prime to $Np$.  
Here we give $M(N,\F)$ its discrete topology and $\End_\F(M(N,\F))$ its compact-open topology.
Then $M=M(N,\F)$ and $\Fc=\Fc(N,\F)$ are topological $A$-modules. Note that if $f \in M$ (or if $f \in \Fc$) the sub-module $Af$ of $M$ (respectively of $\Fc$) generated by  $f$ is finite-dimensional over $\F$ by \ref{locfin}, and hence is finite as a set.

By construction, the maximal ideals of $A(N,\F)$ correspond to the $\Gal(\bar \F_p/\F)$-conjugacy classes of systems of eigenvalues 
(for the $T_\ell$ and $S_\ell$, $\ell \nmid Np$) appearing in $M(N,\bar \F_p)$.
As recalled earlier, the set of such systems is finite, and  in natural bijection (determined by the Eichler-Shimura relation~\ref{es}) with the set $R(N,p)$.  Further, by our choice of $\F$, all those eigenvalues are in $\F$. It follows that $A$ is a semi-local ring; more precisely that we have a natural decomposition
$$ A = \prod_{\rhob \in R(N,p)} A_\rhob$$
where $A_\rhob$ is the localization of $A$ at the maximal ideal corresponding to the system of eigenvalues corresponding to $\rhob$. The quotient $A_\rhob$ of $A$ is a complete local $\F$-algebra of residue field $\F$, and if one denotes by $T_\rhob$ the image of an element $T \in A$ in  $A_\rhob$, then $A_\rhob$ is characterized among the local components of $A$ by the following property. 
\begin{num} \label{Tellrhb} For every $\ell \nmid Np$, the elements $T_{\ell,\rhob} - \tr \rhob(\Frob_\ell)$ and $\ell S_\ell - \det \rhob(\Frob_\ell)$ belong to the maximal ideal $\mm_\rhob$ of $A_\rhob$ (or equivalently, are topologically nilpotent in $A_\rhob$).\end{num}

\par \bigskip

The decomposition of $A$ as $\prod A_ \rhob$ gives rise to corresponding decompositions of $M=M(N,\F)$ and $\Fc=\Fc(N,\F)$:
$$M = \oplus_{\rhob \in R(N,p)} M_\rhob,\ \  \Fc = \oplus_{\rhob \in R(N,p)} \Fc_\rhob,$$
such that $A_\rhob M_\rhob = M_\rhob$ and $A_{\rhob} M_{\rhob'}=0$ if $\rhob \neq \rhob'$, and similarly for $\Fc$.
In other words, $M_\rhob$ (or $\Fc_\rhob$) is the common generalized eigenspace in $M$ (respectively $\Fc$) for all the operators $T_\ell$ and $\ell S_\ell$ 
($\ell \nmid Np)$ with generalized eigenvalues $\tr \rhob(\Frob_\ell)$ and $\det \rhob(\Frob_\ell)$.

Let $\rhob \in R$. Since $A$ acts faithfully on $M$, the algebra $A_\rhob$ acts faithfully on $M_\rhob$. In particular $M_\rhob$ is non-zero. It is easy to deduce that $M_\rhob$ contains a non-zero eigenform for the all the Hecke operators $T_\ell$ and $S_\ell$, $\ell \nmid Np$. We shall need in one occasion the following slightly more precise result, due to Ghitza \cite{ghitza}.

\begin{num} \label{cuspidaleigenform} Let $\rhob \in R$. There exists a form $f=\sum_{n=1}^\infty a_n q^n$ in $M_\rhob$ with $a_0=0$, $a_1=1$ and that is an eigenform for all the Hecke
operators $T_\ell$ and $S_\ell$, $\ell \nmid Np$.
\end{num}
Indeed, according to \cite[Theorem 1]{ghitza}, there exists an eigenform $h \in M_\rhob$ which is cuspidal, that is such that $a_0(h)=0$. Let $m \in \N$ such that $a_m(h) \neq 0$. Then $f=\frac{1}{a_m(h)} U_m h$ is an eigenform and satisfies $a_0(f)=0$, $a_1(f)=1$.

\subsection{The Hecke modules $Af$ and the Hecke algebra $A_f$}

Recall that if $f \in M(N,\F)$, we defined $Af$ to be the submodule of $M$ (over $A$) 
generated by $f$, which by \ref{locfin} is a finite-dimensional vector space over $\F$.
We shall denote by $A_f$  the image of $A$ under the restriction map $\End_{\F}(M) \rightarrow \End_{\F}(Af)$.
Thus $A_f$ is a finite dimensional quotient of $A$. We continue to denote by $T_\ell$ and $S_\ell$ the images 
of $T_\ell$ and $S_\ell$ in $A_f$.

\subsection{The support $R(f)$ of a modular form}

 For $f \in M$, we define the support of $f$ to be the subset of $R$ consisting of those 
representations $\rhob$ such that the component $f_\rhob$ of $f$ in 
$M_\rhob$ is non-zero.  We will denote the support of $f$ by $R(f)$.  Thus $R(f) = \emptyset$ if and only if $f=0$, and $R(f)$ is a singleton $\{\rhob\}$ if and only if $f$ is a generalized eigenform for all the operators $T_\ell$ and $S_\ell$ (with $\ell \nmid Np$). 
Equivalently, $R(f)$ is the smallest subset of $R$ such that the natural surjection $A=\prod_{\rhob \in R} A_\rhob \rightarrow A_f$ factors through $\prod_{\rhob \in R(f)} A_\rhob$.  In view of \ref{Tellrhb}, we have the following lemma. 

\begin{lemma} \label{tgnil} Let $\ell \nmid Np$. The action of the operator $T_\ell$ on the finite-dimensional space $Af$ is nilpotent 
if and only if $\Frob_\ell \in N_\rhob$  for every $\rhob \in R(f)$.  Similarly, the action of $R_\ell$ on $Af$ is invertible if and only if $\Frob_\ell \in U_\rhob$ for every $\rhob\in R(f)$.  
\end{lemma}

\subsection{Pure modular forms and the invariants $\alpha(f)$ and $h(f)$}

\label{puredef}

\begin{definition} We say that $f \in M$ is {\it pure}  if for every $\rhob, \rhob' \in R(f)$, one has $N_\rhob=N_{\rhob'}$, or equivalently $U_\rhob=U_{\rhob'}$.  If $f$ is pure, and non-zero, we denote by $N_f$ and $U_f$ the common sets $N_{\rhob}$ and $U_{\rhob}$ for $\rhob \in R(f)$.   Further, we let ${\mathcal N}_f$ and ${\mathcal U}_f$ denote the sets of primes $\ell \nmid Np$ with $\Frob_\ell \in N_f$ and $\Frob_\ell \in U_f$ respectively.  
\end{definition}

Note that generalized eigenforms are pure, but that the converse is false in general.
 Also note that, by Lemma \ref{tgnil}, if $f$ is non-zero and pure, and $\ell \nmid Np$ then $T_\ell$ is 
 nilpotent on $Af$ if $\ell \in \Nc_f$, and $T_\ell$ is invertible on $Af$ if $\ell \in \Uc_f$.  

\begin{definition} Let $f$ be a pure, non-zero, modular form. We define $\alpha(f)=\mu_{G_{\Q,Np}}(N_f)$, so that $\alpha(f)=\alpha(\rhob)$ for any $\rhob \in R(f)$.
We define the {\it strict order of nilpotence} of $f$,  denoted by $h(f)$, as the largest integer $h$ such that there exist (not necessarily distinct) prime numbers $\ell_1,\dots,\ell_h \nmid Np$ in $\Nc_f$ with $T_{\ell_1} \dots T_{\ell_h} f \neq 0$.
\end{definition}

Note that in the definition of the strict order of nilpotence, the largest integer $h$ exists and is no more than the dimension of $Af$, since the $T_{\ell_i}$ act nilpotently on $Af$ for $\ell_i \in \Nc_f$.

\begin{num}  \label{decfpure} Given a general non-zero form $f$, partition the finite set $R(f)$ into equivalence classes $R_i(f)$ based on the equivalence relation $\rhob\sim\rhob'$ if and only if $N_\rhob=N_\rhob'$.  Thus we may write 
$$
f = \sum_i f_i,\ \ \ f_i=\sum_{\rhob \in R_i(f)} f_\rhob,
$$
 so that the $f_i$ are pure. We call this decomposition the {\rm canonical decomposition of $f$ into pure forms}.
 \end{num}

We now extend the definition of $\alpha(f)$ and $h(f)$ to forms that are not necessarily pure. 

\begin{definition} If $f=\sum_i f_i$ is the canonical decomposition of $f$ into pure forms,
we set $\alpha(f) = \min_i \alpha(f_i)$, and $h(f)=\max_{i, \alpha(f_i)=\alpha(f)} h(f_i)$.
\end{definition}

\subsection{Existence of a pseudorepresentation and consequences} 

\begin{prop}
There exist continuous maps $t: G_{\Q,Np} \rightarrow A$, $d: G_{\Q,Np} \rightarrow A$ such that
\begin{itemize}
\item[(i)] $d$ is a morphism of groups $G_{\Q,Np} \rightarrow A^\ast$.
\item[(ii)] $t$ is central (i.e. $t(gh)=t(hg)$)
\item[(iii)] $t(1)=2$.
\item[(iv)] $t(gh)+t(gh^{-1})d(h)=t(g) t(h)$ for all $g,h \in G_{\Q,Np}$
\item[(v)] $t(\Frob_\ell)=T_\ell$ for all $\ell \nmid Np$.
\item[(vi)] $d(\Frob_\ell)=\ell S_\ell$ for all $\ell \nmid Np$.
\end{itemize}
\end{prop}
 The uniqueness of such a pair $(t,d)$ is clear: the function $t$ is
characterized uniquely by (ii) and (v) alone using the Chebotarev density theorem, and $d$ is characterized by (i) and (vi) (or else by (iv), see (\ref{dt}) in Remark~\ref{remarktd} below). The existence of $t$ and $d$ is proved by ``glueing"
the traces and determinants of the representations attached by Deligne to eigenforms in characteristic zero, and then reducing modulo $p$. For details, see~\cite{BK}.

\begin{remark} \label{remarktd}
The properties (i) to (iv) express the fact that $(t,d)$ is a pseudo-representation of dimension $2$. The map $t$ is called the trace, and the map $d$ is called the determinant of the representation $(t,d)$, cf. \cite{chenevier}. It is easy to check that the trace and determinant of any continuous two dimensional representation (of a topological group over any topological commutative ring) satisfy properties (i) to (iv). Since $p>2$, one can recover $d$ from $t$ by the formula 
\begin{eqnarray} \label{dt} d(g)=(t(g)^2-t(g^2))/2, \end{eqnarray} which follows upon taking $g=h$ in (iv) and using (iii). 
\end{remark}

We prove for later use the following lemma.
\begin{lemma} \label{tgp} For every $g \in G_{\Q,Np}$ one has $t(g^p)=t(g)^p$.
\end{lemma}
\begin{pf} Let $m \in \GL_2(A)$ be the matrix $\mat{ 0 & -1 \\ d(g) & t(g)}$ so that $\tr(m)=t(g)$ and $\det(m)=d(g)$. Since the function $\tr$ and $\det$ 
on the multiplicative subgroup generated by $m$ satisfy properties (i) to (iv) above, one sees easily by induction on $n$ that $\tr(m^n)=t(g^n)$ for all $n$.
Thus it suffices to prove that $\tr(m^p)=\tr(m)^p$. 

Let  $f:\F_p[D,T] \rightarrow A$ be the morphism of rings sending $D$ to $d(g)$ and $T$ to $t(g)$, where $D$ and $T$ are two indeterminates. Let $M \in \GL_2(\F_p[D,T])$ be the matrix $\mat{ 0 & -1 \\ D & T}$. Since $f(M)=m$, it suffices clearly to prove that $\tr(M^p)=\tr(M)^p$. Since $\F_p[D,T]$ can be embedded in an algebraic field $k$ of characteristic $p$, it suffices to prove that for all $M \in M_2(k)$, $\tr(M^p)=\tr(M)^p$.  Replacing $M$ by a conjugate matrix if necessary, we may assume that $M$ is triangular, in which case the formula is obvious.
\end{pf}


Let $f \in M(N,\F)$ be a modular form. 
Let $t_f:G \rightarrow A_f$ and 
$d_f: G \rightarrow A_f$ be the composition of $t$ and $d$ with the natural morphism of algebras $A \rightarrow A_f$. Note that $(t_f,d_f)$ satisfies the same 
properties (i) to (vi), and so in particular, $(t_f,d_f)$ is a pseudo-representation of $G$ on $A_f$. In particular, (v) reads
\begin{eqnarray} \label{tfl} t_f(\Frob_\ell) f = T_\ell f. \end{eqnarray}

We now deduce certain consequences of the existence of the pseudo-representation $(t,d)$ for the algebra $A$ and for modular forms $f \in M$. 
\begin{prop}\label{generatorsHecke} The Hecke algebra $A$ is topologically generated by the $T_\ell$ for $\ell \nmid Np$ alone (that is, without the $S_\ell$).
\end{prop}
\begin{pf} 
Let $A'$ be the closed sub-algebra of $A$ generated by the $T_\ell$. Since the elements $\Frob_\ell$ for $\ell \nmid Np$ are dense in $G_{\Q,Np}$, and $t(\Frob_\ell) = T_\ell \in A'$,
one sees that $t(G_{\Q,Np}) \subset A'$. In particular, for $\ell$ not dividing $Np$, $t(\Frob_\ell^2) \in A'$, hence also $(t(\Frob_\ell^2)-t(\Frob_\ell)^2)/2$. But this element is just $d(\Frob_\ell)=\ell S_\ell $.  Hence $S_\ell \in A'$ and $A'=A$.
\end{pf}

\begin{lemma} \label{finitequot} There exists a finite quotient $G_f$ of $G_{\Q,Np}$ such that for $\ell \nmid Np$, the action of $T_\ell $ on $Af$ 
depends only on the image of $\Frob_\ell$ in $G_f$.
\end{lemma}
\begin{pf} 
Let $H$ denote the subset of $G_{\Q,Np}$ consisting of elements $h$ such that $t_f(gh)=t_f(g)$ 
for every $g \in G$.  Since $t$ is central (property (ii) above), it follows that $H$ is a normal subgroup of $G$.  
We call $H$ the {\it kernel} of the pseudo-representation $(t_f,d_f)$. By (\ref{dt}) and (iii) one has $d_f(h)=1$ for $h \in H$. Let $G_f = G_{\Q,Np}/H$. The maps
$t_f, d_f: G_{\Q,Np} \rightarrow A_f$ factor through $G_f$ to give maps $G_f \rightarrow A_f$, which we shall also denote by $t_f$ and $d_f$.  Note that by construction, there is no $h \neq 1$ in $G_f$ such that $t_f(gh)=t_f(g)$ for every $g \in G_f$. Since $A_f$ is finite, it follows easily that $G_f$ is a finite group.
Finally, by (\ref{tfl}), $T_\ell f$ depends only on $t_f(\Frob_\ell)$, which only depends on the image of $\Frob_\ell$ in $G_f$. Therefore if $g \in Af$, then $g=Tf$ for some $T \in A$ and $T_\ell g = T_\ell T f = T T_\ell f$ depends only on  the image of $\Frob_\ell$ in $G_f$.
\end{pf}

We draw three consequences of this lemma. 

\begin{prop}\label{squarefreenonzero} Let $f = \sum_{n=0}^\infty a_n q^n \in \Fc=\Fc(N,\F)$. If $f \neq 0$, there exists a square-free integer $n$ such that $a_n \neq 0$.
\end{prop} 
\begin{pf} Since $f$ is non-zero, $a_n\neq 0$ for some $n\in \N$, and since $f\in \Fc$ one has $(n,Np)=1$. Thus
$a_1(T_n f)\neq 0$. By Proposition~\ref{generatorsHecke}, $T_n$ is a limit of linear combinations of terms of the form $T_{\ell_1} \dots T_{\ell_s}$ with $\ell_1,\dots,\ell_s$ being (not necessarily distinct) primes all not dividing $Np$. Since $T \mapsto a_1(Tf)$ is continuous and linear, we deduce that $a_1(T_{\ell_1} \cdots T_{\ell_s} f) \neq 0$ for some primes $\ell_1, \dots \ell_s$ not dividing $Np$ (again not necessarily distinct). Since the action of $T_{\ell_i}$ on $Af$ depends only on $\Frob_{\ell_i}$ in the finite Galois group $G_f$, one can replace $\ell_i$ by any other prime whose Frobenius has the same image without affecting the action of $T_{\ell_i}$.  In this manner, we may find distinct primes $\ell_i^{\prime}$ such that 
$T_{\ell_1}\cdots T_{\ell_s} = T_{\ell_1^{\prime}} \cdots T_{\ell_s^{\prime}}$, and then with $m=\ell_1^{\prime}\cdots \ell_s^{\prime}$ it follows that $a_m(f)=a_1(T_{m} f ) =a_1(T_{\ell_1'}\cdots T_{\ell_s'} f) = a_1(T_{\ell_1} \dots T_{\ell_s} f) \neq 0$.
\end{pf}

\begin{prop} \label{multifrob} Let $f\in M(N,\F)$ be a pure form, and let $f^{\prime}$ be any element of $M(N,\F)$.  Let $h$ be a non-negative integer, and let $\Mc$ denote the set of square-free integers $m$ having exactly $h$ prime factors, all from the set $\Nc_f$, and such that $T_m f = f'$. Then $\Mc$ is multi-frobenian.
\end{prop}
\begin{pf} Let $G_f$ be as in Lemma~\ref{finitequot} and let $D_{f,f'} \subset G_f^h$ denote the set of $h$-tuples 
$(g_1,\dots,g_h)$ such that $t_f(g_1) \dots t_f(g_h) f = f'$ and with all the $g_i \in N_f$.
Then $D_{f,f'}$ is invariant under conjugation and symmetric under permutations, and 
hence by definition $\Mc$ is the multi-frobenian set of weight $h$ attached to $D_{f,f'}$ and $G_f$.
\end{pf}

\begin{prop} \label{distinct} Let $f$ be a pure modular form. Then there exist $h(f)$ {distinct} primes $\ell_1,\dots,\ell_{h(f)}$ in $\Nc_f$ such that 
$T_{\ell_1} \dots T_{\ell_{h(f)}} f  \neq 0$.
\end{prop}
\begin{pf} That we can find $h(f)$ primes $\ell_1,\dots,\ell_{h(f)}$ in $\Nc_f$ such that 
$f' := T_{\ell_1} \dots T_{\ell_{h(f)}} f  \neq 0$ is just the definition of $h(f)$. In the notation of the previous  proposition we see that $D_{f,f'}$ is not empty as it contains $(\Frob_{\ell_1},\dots,\Frob_{\ell_{h(f)}})$.  Hence the multi-frobenian set $\Mc$ of that proposition is not empty, and there exist {distinct} 
primes $\ell'_1,\dots,\ell'_{h(f)}$ in $\Nc_f$ such that $T_{\ell'_1} \dots T_{\ell'_{h(f)}} f  = f'\neq 0$.
\end{pf}
\section{Asymptotics: Proof of Theorem \ref{main}} 


Let $f=\sum a_n q^n \in M=M(N,\F)$. We assume below that $f$ is not constant.   We set
$$ 
Z(f) = \{ n\in \N, a_n \neq 0\} \qquad \text{and} \qquad \pi(f,x) = |\{ n < x,  a_n \neq 0\}|, 
$$ 
and our goal is to establish an asymptotic for $\pi(f,x)$.   For a given $a\in \F^*$ it will also be convenient to define 
$$ 
Z(f,a) = \{ n\in \N, a_n =a\} \qquad \text{and} \qquad \pi(f,a,x) = |\{ n< x, a_n = a\}|. 
$$ 
By \ref{levelred}, we may assume without loss of generality that $(N,p)=1$, so all the results of \S\ref{heckealgebra} apply.


\subsection{Proof of Theorem~\ref{main} when $f \in \Fc(N,\F)$ and  $f$ is pure}
\label{proofasymppure}

We assume in this section  that $f$ is a pure form in $\Fc(N,\F)$. 
From~\S\ref{puredef} recall that the set of primes $\ell$ not dividing $Np$ may be partitioned into two sets, $\Uc_f$ and $\Nc_f$, such that $\ell \in \Uc_f$ if $T_\ell$ acts invertibly on $Af$ and $\ell \in \Nc_f$ if $T_\ell$ acts nilpotently on $Af$.  

Given $a\in \F^*$ we wish to prove an asymptotic for $\pi (f,a,x)$.  If $n$ is an integer with $a_n(f)= a$ (and since $f\in \Fc$ we 
must have $(n,Np)=1$) then we may write $n=m m^{\prime} m^{\prime\prime}$ with $m$ square-free and containing all prime factors from $\Uc_f$, $m^{\prime}$ square-free with $h\le h(f)$ prime factors all from $\Nc_f$, and with $m^{\prime\prime}$ square-full  
and coprime to $mm^{\prime}$.   Such a decomposition of the number $n$ is unique, and if we write $f''=T_{m''}f$ and $f' = T_{m'} f''$ then $f'$ and $f''$ are forms in $Af -\{0\}$ with $a_m(f')=a$.  Thus integers $n$ with $a_n(f)=a$ define uniquely triples $(f',f'',h)$ and we may decompose 
\begin{equation}
 \label{Zfapart}  Z(f,a) = \coprod_{f',f'',h} Z(f,a;f',f'',h), 
 \end{equation}
where the disjoint union is taken over forms $f'$, $f'' $ in $Af-\{0\}$ and integers $0\le h\le h(f)$.   Here the set $Z(f,a;f',f'',h)$ is 
defined as the set of integers $n=mm'm''$ with $(n,Np)=1$ such that   
\begin{num} \label{Zsn1} $m$ is square-free and all its prime factors are in $\Uc_f$; \end{num} 
\begin{num} \label{Zsn2} $m'$ is square-free, has exactly $h$ prime factors, and all its prime factors are in $\Nc_f$, and 
moreover $f'=T_{m'} f''$; \end{num} 
\begin{num} \label{Zsn3} $m''$ is square-full, relatively prime to $mm'$, and $f''=T_{m''}f$; \end{num}
\begin{num} \label{Zsn4} $a_m(f') = a$. 
\end{num}



 
Next we evaluate the number of elements up to $x$ in the set $Z(f,a;f',f'',h)$ using Theorem~\ref{delangemethod}.
Write $\Sc_{f,f''}$ for the set of square-full integers $m''$ such that $T_{m''} f = f''$, and $\Mc_{f',f''}$ for the set of integers $m'$ that are the product of $h$ distinct primes in 
$\Nc_f$ and such that $f'=T_{m'}  f''$. By Proposition~\ref{multifrob}, $\Mc_{f',f''}$ is a multi-frobenian set of height $h$. 
Observe that conditions \ref{Zsn1}, \ref{Zsn2}, \ref{Zsn3} are the same as the conditions  \ref{Zcn1}, \ref{Zcn2}, \ref{Zcn3} defining the set $\Zc(\Uc_f,\Mc_{f',f''},\Sc_{f,f''})$. Now, define a map $\tau_f: \Uc_f \rightarrow A_f^\ast$ sending $\ell$ to $t_f(\Frob_\ell) = T_\ell$ and extend it by multiplicativity to the set of all square-free integers composed only of primes from $\Uc_f$.   Let $\Gamma_f$ be the image of $\tau_f$, which is a finite abelian subgroup of the finite group $A_f^*$, and let $\Delta_{f',a}$ denote the set of $\gamma \in \Gamma_f$ such that $a_1(\gamma f')=a$.   For $n =mm'm'' \in Z(f,a;f',f'',h)$ put $\tau_f(n) = \tau_f(m)$, and so the condition 
\ref{Zsn4} is the same as $\tau_f(n) \in \Delta_{f',a}$.   Thus we are in a position to apply Theorem \ref{delangemethod}, which yields, 
%
%
assuming that the sets $\Mc_{f',f''},\Sc_{f,f''}$ and $\Delta_{f',a}$ are all not empty, 
\begin{align} \label{densZfaffh}
\nonumber |\{n<x: n \in Z(f,a; f',f'',h)\} | 
&= | \{n<x: n\in \Zc(\Uc_f,\Mc_{f',f''},\Sc_{f,f''}), \tau(n) \in \Delta_{f',a}\} | \\ 
&\sim 
 c  \, \delta(\Mc_{f',f''})  \frac{|\Delta_{f',a}|}{|\Gamma_f|}\frac{x}{(\log x)^{\alpha(f)}}(\log \log x)^h, 
 \end{align}
where $c=c(f,f'')>0$ is a constant depending only on $\Uc_f$ and $\Sc_{f,f''}$ (thus only on $f$ and $f''$), and $\alpha(f)=1-\delta(\Uc_f)=\delta(\Nc_f)$ as defined in \S\ref{puredef}. If at least one of the sets $\Mc_{f',f''}$, $\Sc_{f,f''}$ or $\Delta_{f',a}$ is empty, then so is $Z(f,a;f',f'',h)$.

Using~(\ref{Zfapart}), one deduces that either all the $Z(f,a,f',f'',h)$ are empty for all permissible choices of 
 $(f',f'',h)$, in which case  $\pi(f,a,x)=0$ for all $x$,
or
\begin{eqnarray} \label{pifax} \pi(f,a,x) \sim c(f,a) \frac{x}{(\log x)^{\alpha(f)}}(\log \log x)^{h(f,a)}, \end{eqnarray}
where $h(f,a) \leq h(f)$ is the  largest integer $h \leq h(f)$ for which there exists $f',f'' \in Af-\{0\}$ such that $Z(f,a;f',f'',h)$ is not empty,
and
\begin{eqnarray}\label{cpfa} c(f,a) = \sum_{(f',f'',h(f,a))} c(f,f'') \delta(\Mc_{f',f''})  \frac{\#\Delta_{f',a}}{\#\Gamma_f}, \end{eqnarray}
the sum being  over those $f',f'' \in Af-\{0\}$ such that $Z(f,a;f',f'',h(f,a))$ is not empty.

We claim that the set $Z(f,a;f',f'',h(f))$ is not empty for some choice of $(f',f'') \in (Af-\{0\})^2$ and some $a \in \F^\ast$. To see this, 
take $m''=1$ and $f''=f=T_{m''} f$. By Proposition~\ref{distinct}, there exists an integer $m'$ with $h(f)$ distinct prime factors in $\Nc_f$
such that $T_{m'} f \neq 0$. Fix one such $m'$ and let $f'=T_{m'} f$. Proposition~\ref{squarefreenonzero} tells us that there exists a square-free integer 
$m$ such that $a_m(f') \neq 0$. Note that $h(f') =0$, hence $m$ has all its prime factors in $\Uc_f$. Define $a=a_m(f') \in \F^\ast$. Then the set $Z(f,a;f',f'',h(f))$ contains $n=mm'm''$ and is therefore not empty, which proves the claim.

Since $\pi(f,x)=\sum_{a \in \F^\ast} \pi(f,a,x)$, it follows from (\ref{pifax}) and the above claim that
$$
\pi(f,x) \sim c(f) \frac{x}{(\log x)^{\alpha(f)}}(\log \log x)^{h(f)}
$$
with 
$$c(f) = \sum_{\substack{a \in \F^\ast \\ h(f,a)=h(f)}} c(f,a).$$

\subsection{Proof of Theorem~\ref{main} when $f \in \Fc(N,\F)$ but $f$ is not necessarily pure}
\label{proofasymp2}
Let  $f = \sum_i f_i$ be the canonical decomposition (see \ref{decfpure})  of $f$ into pure forms.
By the preceding section, one has
$$\pi(f_i,x)  \sim c(f_i)  \frac{x}{(\log x)^{\alpha(f_i)}} (\log \log x)^{h(f_i)}.$$
Consider the indices $i$ such that $\alpha(f_i)$ is minimal (and by definition $\alpha(f_i)=\alpha(f)$) and among those select 
those indices with $h(f_i)$ maximal (and by definition $h(f_i)=h(f)$); let $I$ denote the set of such indices.  
We claim that
$$\pi(f,x) \sim c(f)   \frac{x}{(\log x)^{\alpha(f)}} (\log \log x)^{h(f)}, \text{ with }c(f)=\sum_{i \in I} c(f_i).$$
To prove the claim, first note that we can forget those $f_i$ with $i \not \in I$, 
because they have a negligible contribution compared to the asserted asymptotic 
(either the power of $\log \log x$ is smaller, or the power of $\log x$ is larger).   It remains to prove that for $i,j \in I$, $i \neq j$, one has
\begin{eqnarray} \label{pififj} \pi(f_i,f_j,x)
=o\left( \frac{x}{(\log x)^{\alpha(f)}} (\log \log x)^{h(f)}\right), \end{eqnarray}
where
$\pi(f_1,f_j,x) = |\{n \leq x, a_n(f_i) \neq 0, a_n(f_j) \neq 0\}|$.  
But if $n$ is such that $a_n(f_i) \neq 0$ and $a_n(f_j) \neq 0$, it has at most $h(f_i)+h(f_j)=2h(f)$ prime factors $\ell$ such that $\Frob_\ell \in N_{f_i} \cup N_{f_j}$.
Moreover, the two open sets $N_{f_i}$ and $N_{f_j}$ of $G_{\Q,Np}$ are not equal by definition of the decomposition into pure forms \ref{decfpure}. Therefore the measure $\alpha'$ of the open set $N_{f_i} \cup N_{f_j}$ is strictly greater than the common measure $\alpha(f)=\alpha(f_i)=\alpha(f_j)$ of $N_{f_i}$ and $N_{f_j}$.
Hence an application of Theorem~\ref{delangemethod} gives
$$
\pi(f_1,f_j,x) = O \left(  \frac{x}{(\log x)^{\alpha'}} (\log \log x)^{2 h(f)}\right),
$$
which implies $(\ref{pififj})$ since $\alpha'>\alpha(f)$.

\subsection{Proof of Theorem~\ref{main}: general case}

\label{pfasympgeneralcase}
Let $\Bc$ be the set of integers $m \geq 1$ all of whose prime factors divide $Np$. Note that the series $\sum_{m \in \Bc} \frac{1}{m}$ converges.
For $m \in \Bc$, we consider the following operators on $\F[[q]]$: 
$$
 U_m \Big(\sum a_n q^n\Big) = \sum a_{m n} q^n, \qquad \text{and} \qquad 
V_m \Big(\sum a_n q^n\Big) = \sum a_{n} q^{m n}.
$$ 
We also consider the operator $W$, defined by
$$
W \Big(\sum a_n q^n\Big) = \sum_{\substack{n \\(n,Np)=1}} a_n q^n.
$$
The operators $U_m$  stabilize the space $M(N,\F)$, see \S\ref{secHeckeonM}. The operator $V_m$ however does not stabilize $M(N,\F)$ (except for $m = p$, see \ref{Vp}), but it sends $M(N,\F)$
 into $M(Nm,\F)$ since it is the reduction mod $p$ of the action on $q$-expansions of the operator on modular forms $f(z) \mapsto f(m z)$.
 As for the operator $W$, it is easily seen from the definitions to satisfy
 $$W = \sum_{m \in \Bc} \mu(m) V_m U_m,$$ 
where $\mu(m)$ is the Möbius function. Since $\mu$ vanishes on non-square-free integers, the sum is in fact finite, and it follows that $W$ sends $M(N,\F)$ into $M(N^2,\F)$, and more precisely into $\Fc(N^2,\F)$.

Let $f= \sum a_n q^n \in M(N,\F)$ be a modular form. For any integer $m \in \Bc$, define
$$
f_m = \sum_{\substack{ n = m m' \\ { (m',Np)=1}}} a_n q^n,
$$
so that $f = a_0 +  \sum_{m \in \Bc} f_m$.  
This sum may genuinely be infinite, but it obviously converges in $\F[[q]]$.   Clearly 
$$
\pi(f,x) = \sum_{m \in \Bc} \pi(f_m,x) + O(1),
$$
where the error term $O(1)$ is just $0$ if $a_0 = 0$ and $1$ otherwise.
One sees from the definitions that $f_m = V_m W U_m f$, so that 
$$
\pi(f_m,x) = \pi(W U_m f, x/m).  
$$
Since $\pi(f_m,x)$ is clearly at most $x/m$, and as $\sum_{m \in \Bc, m> (\log x)^2} 1/m \ll 1/\log x$, we 
conclude that 
\begin{eqnarray} \label{pifxassum}
\pi(f,x) = \sum_{\substack{ m \in \Bc \\  m\le (\log x)^2}} \pi(WU_m f,x/m) + O\Big(\frac{x}{\log x}\Big).
\end{eqnarray}

Now $W U_m f  \in \Fc(N^2,\F)$, and we can apply the results of \S\ref{proofasymp2} and thus estimate $\pi(WU_m f,x/m)$.  Thus,   if 
$WU_m f \neq 0$, and $m\le (\log x)^2$ (so that $\log (x/m) \sim \log x$)  
\begin{eqnarray} \label{piUmf}
\pi(W U_m f, x/m) \sim c(W U_m f) \frac{x}{m(\log x)^{\alpha(W U_m f)}} (\log \log x)^{h(W U_m f)}.
 \end{eqnarray} 
 Note that since $f$ is not a constant, $W U_m f \neq 0$ for at least one $m \in \Bc$.
Further, note that while $\Bc$ is infinite, the set of forms $W U_m f$ for $m \in \Bc$ is finite since $U_m f$ belongs to the Hecke-module generated by $f$ which is finite-dimensional over $\F$ (see \ref{locfin}).  Thus the asymptotic formula \eqref{piUmf} holds uniformly for all $m\le (\log x)^2$ with $m\in \Bc$ and as $x\to \infty$.   Finally, since the Hecke operators $T_\ell$ for $\ell$ prime to $Np$ commute with the operators $U_m$, $V_m$ and $W$, it follows that  
$$
\alpha(f) = \min_{\substack{m \in \Bc \\  WU_m f \neq 0}} \alpha( W U_m f), 
\qquad \text{and} \qquad h(f)=\max_{\substack{ m \in \Bc \\   WU_m f \neq 0 \\ \alpha(W U_m f) = \alpha(f)}} h(W U_m f).
$$
Thus, putting $c_m = c(WU_mf)$ when $WU_mf\neq 0$ (which 
happens for at least one $m\in \Bc$) and putting $c_m=0$ otherwise, we may recast (\ref{piUmf}) as 
\begin{eqnarray} \label{piUmf2}
\pi(W U_m f, x/m) = (c_m + \epsilon_m(x)) \frac{x}{m(\log x)^{\alpha( f)}} (\log \log x)^{h(f)}, \end{eqnarray}
where $\epsilon_m(x)\to 0$ as $x \rightarrow \infty$, uniformly for all $m\in \Bc$ with $m\le (\log x)^2$.  


From (\ref{pifxassum}) and (\ref{piUmf2}) we obtain 
$$
\pi(f,x) \sim \sum_{\substack{ m\in \Bc \\ m <(\log x)^2}} \frac{c_m}{m}  \frac{x}{(\log x)^{\alpha( f)}} (\log \log x)^{h(f)} \sim c
\frac{x}{(\log x)^{\alpha(f)}} (\log \log x)^{h(f)},
$$
 with 
 \begin{eqnarray} \label{defccm} 
 c = \sum_{m \in \Bc} \frac{c_m}m, 
 \end{eqnarray} 
 noting that this series converges because $c_m$ takes only finitely many values (and hence is bounded).  This 
 finishes the proof of Theorem \ref{main}. 

 \section{Equidistribution}
 \label{equidis}
 \begin{definition} We say that a form $f \in M(\Gamma_1(N),\F)$ has the {\it equidistribution property} if for any two $a,b \in \F^\ast$, we have  $\pi(f,a,x) \sim \pi(f,b,x)$.
 We say that a subspace $V \subset M(\Gamma_1(N),\F)$ has the {\it equidistribution property} if every non-constant form $f \in V$ has the 
 equidistribution property.
 \end{definition}
 In view of Theorem~\ref{main}, to say that $f$ has the equidistribution property is equivalent to 
 $$
 \pi(f,a,x) \sim \frac{c(f)}{|\F|-1} \frac{x}{\log(x)^{\alpha(f)}} (\log \log x)^{h(f)},
 $$
 where $c(f)$ is the constant of Theorem~\ref{main}.
 
 We now give a sufficient condition for equidistribution for generalized eigenforms, which generalizes a similar criterion for true eigenforms due to Serre (\cite[Exercise 6.10]{SerreE}).
 
 \begin{prop} \label{equi1} Let $\rhob: G_{\Q,Np} \rightarrow \GL_2(\F)$ be a representation in $R(N,p)$. If the  set $ \tr \rhob(G_{\Q,Np}) - \{0\}$ generates $\F^\ast$ multiplicatively, then the generalized eigenspace $M(N,\F)_{\rhob}$
 has the equidistribution property. \end{prop}
\begin{pf} 
First assume that $f \in \Fc(N,\F)_{\rhob}$. Since $f$ is pure, the asymptotic formula (\ref{pifax}) holds for $\pi(f,a,x)$, and 
to obtain equidistribution it remains to show that the constant $c(f,a)$ appearing there is independent of $a \in \F^\ast$. By formula (\ref{cpfa}), which gives the values of $c(f,a)$, it suffices to prove that the
cardinalities of the subsets $\Delta_{f',a}$ of $\Gamma_f$ are independent of $a \in \F^\ast$, for any given form $f' \in Af-\{0\}$. Recall that $\Gamma_f$ is the subgroup of $A_f^\ast$ generated by the elements $T_\ell = t_f(\Frob_\ell)$ for $\ell \in \Uc_f = \Uc_{\rhob}$, hence by Chebotarev and the definition of $\Uc_f$, the subgroup of $A_f^\ast$ generated by $t_f(G_{\Q,Np}) \cap A_f^\ast$; recall also that $\Delta_{f',a}$ is the set of elements $\gamma \in \Gamma_f$ such that $a_1(\gamma f')=a$. To prove that $|\Delta_{f',a}|$ is independent of $a$, it therefore
suffices to prove that $\Gamma_f$ contains the subgroup $\F^\ast$ of $A_f^\ast$, in which case multiplication by $ba^{-1}$ will induce a bijection between $\Delta_{f',a}$ and $\Delta_{f',b}$ for any $b \in \F^\ast$. Since by hypothesis $\tr \rhob(G_{\Q,Np})-\{0\}$ generates $\F^\ast$, it suffices to show that 
$\tr \rhob(G_{\Q,Np})-\{0\} \subset \Gamma_f$. For this,  let $g \in G_{\Q,Np}$, and assume that $\tr \rhob(g) \neq 0$. By~\ref{Tellrhb},
one has $t_f(g) \equiv \tr \rhob(g) \pmod{\mm_{A_f}}$ where $\mm_{A_f}$ is the maximal ideal of the finite local algebra $A_f$. Let $n$ be an integer such that $\mm_{A_f}^n=0$, and let $q$ be the cardinality of $\F$. Then by Lemma~\ref{tgp},
$$
t_f(g^{q^n})=t_f(g)^{q^n} \equiv (\tr \rhob(g))^{q^n} \pmod{\mm_{A_f}^n},
$$
so that, since $x \mapsto x^q$ induces the identity on $\F$,
$$t_f(g^{q^n}) = \tr \rhob(g).
$$ 
Hence $\tr \rhob(g) \in \Gamma_f$ and this completes the proof of the proposition for forms $f \in \Fc(N,\F)_{\rhob}$.

Now consider a general non-constant form $f \in \Mc(N,\F)_{\rhob}$. Mimicking the proof in \S\ref{pfasympgeneralcase}, one has 
$\pi(f,a,x)=\sum_{m \in \Bc, m\le (\log x)^2} \pi(W U_m f,a,x/m)+O(x/\log x)$ and the asymptotic obtained for $\pi(WU_m f,a,x/m)$ is independent of $a \in \F^\ast$ since $W U_m f \in \Fc(N^2,\F)_{\rhob}$ and by the result just established. 
This completes the proof. 
\end{pf}

Serre has given an example of an eigenform $f$ mod $p$ that does not have the equidistribution property: namely, 
the form $\Delta$ mod $7$, see \cite[Exercise 12]{SerreE}.  Here is a generalization.

\begin{prop} \label{eigenequi} Let $f$ be a non-constant {\it eigenform} in $\Fc(N,\F)_{\rhob}$.  If the  set $ \tr \rhob(G_{\Q,Np}) - \{0\}$ does not generate $\F^\ast$ multiplicatively,
then $f$ does not have the equidistribution property.
\end{prop}
\begin{pf} Write $f=\sum_{n=1}^\infty a_n q^n$. 
Since $f$ is an eigenform for the $T_\ell$, $\ell \nmid Np$, and also is killed by the $U_\ell$ for $\ell \mid Np$ (because it is in $\Fc$),  the sequence $a_n$ is multiplicative and one has $a_\ell=0$ for $\ell \mid Np$,  and  $a_\ell = \tr \rhob(\Frob_\ell)$ for all $\ell \nmid Np$.  Also one has $a_1 \neq 0$ since $f$ is non-constant, and we may assume $a_1=1$.

Let $B$ be the proper subgroup of $\F^\ast$ generated by $\tr \rhob(G_{\Q,Np})-\{0\}$. By mutiplicativity $a_n \in B \cup \{0\}$ for all square-free integers $m$.  Since $a_n \neq 0$ for square-free $n$ exactly when $n$ is composed only of primes in $\Uc_f$, we see that 
\begin{equation} \label{lowerbound}
\sum_{\substack{ n\le x \\ a_n \in B }} 1 \ge \sum_{\substack{ n\le x \\ n\text{ square-free} \\ p|n \implies p\in \Uc_f }} 
1 \sim c\frac{x}{(\log x)^{\alpha(f)}},
\end{equation} 
for a suitable positive constant $c$.  Now if $f$ has the equidistribution property, then since $|B| \le |\F^* - B|$ 
for proper subgroups $B$ of $\F^*$, we must have 
$$ 
\sum_{\substack{ n\le x \\ a_n \in B }} 1 \le (1+o(1)) \sum_{\substack{ n\le x \\ \ \ \ \ a_n \in \F^* -B }} 1. 
$$
The right hand side above is at most the number of integers of the form $mr \le x$ where $1<m$ is square-full, 
and $r\le x/m$ is square-free with $(r,m)=1$ and $a_r \neq 0$.  Ignoring the condition that $(r,m) =1$, the number of 
such integers is (arguing as in \S \ref{pfasympgeneralcase})
$$ 
\le \sum_{\substack {1 <m \le x \\ m \text{ square-full} } } \sum_{\substack{ r\le x/m \\ r\text{ square-free}\\ p|r \implies p\in \Uc_f}} 1 \le \sum_{\substack {1 <m \le (\log  x)^2 \\ m \text{ square-full} } } \frac{x}{m} \frac{c+o(1)}{(\log x)^{\alpha(f)}} + \sum_{\substack { m>(\log x)^2 \\ m \text{ square-full}}} \frac{x}{m} ,
$$
which is at most 
\begin{align*}
(c+o(1)) \frac{x}{(\log x)^{\alpha(f)}} \sum_{ \substack{1<m \\ m\text{ square-full}}} \frac 1m 
&= (c+o(1)) \frac{x}{(\log x)^{\alpha(f)}} \Big( \frac{\zeta(2)\zeta(3)}{\zeta(6)} -1 \Big) \\
&= (0.9435\ldots c + o(1) ) \frac{x}{(\log x)^{\alpha(f)}}. 
\end{align*}
But this contradicts the lower bound \eqref{lowerbound}, completing our proof.  
\end{pf}

We can use the above result to give a converse to Proposition~\ref{equi1} when the level $N$ is 1.
\begin{prop} \label{recequi} Let $\rhob \in R(1,\F)$. The space $M(1,\F)_{\rhob}$ has the equidistribution property if and only if  the  set $ \tr \rhob(G_{\Q,p}) - \{0\}$ generates $\F^\ast$ multiplicatively.
\end{prop}
\begin{pf}  By \ref{cuspidaleigenform}, $M(1,\F)_{\rhob}$ has an eigenform $f=\sum_{n=1}^\infty a_n q^n$ with $a_1=1$ for all the Hecke operators $T_\ell$ and $S_\ell$, $\ell \neq p$. Replacing $f$ by $f-V_pU_pf$ (see \ref{Vp}), we may assume that $f$ is an eigenform in $\Fc(1,\F)_\rhob$. 
If $M(1,\F)_{\rhob}$, hence $f$, has the equidistribution property, then by the preceding proposition  $ \tr \rhob(G_{\Q,p}) - \{0\}$ generates $\F^\ast$ multiplicatively.
\end{pf}

In the same spirit, but concerning forms that are not necessarily generalized eigenforms, one has the following 
partial result.  

 \begin{prop} If $2$ is a primitive root modulo $p$, then $M(N,\F_p)$ has the equidistribution property.
 \end{prop}
 \begin{proof} 
 One reduces to the case of an $f \in \Fc(N,p)$ pure exactly as in \S\ref{proofasymp2}. Then, arguing as in the proof of Proposition~\ref{equi1}, 
 it suffices to prove that the group $\Gamma_f$ generated by $t_f(G_{\Q,Np})$
 contains $\F_p^\ast$. But $\Gamma_f$ contains $t_f(1)=2$ which by hypothesis generates $\F_p^\ast$.
 \end{proof}
 Again, one has a partial converse to this proposition.
 \begin{prop} In the case $N=1$ and $p \equiv 3 \pmod{4}$, $M(1,\F_p)$ has the equidistribution property if and only if $2$ is a primitive root modulo $p$.
 \end{prop} 
 \begin{proof} Let $\omega_p : G_{\Q,p} \rightarrow \F_p^\ast$ be the cyclotomic character modulo $p$, and let $\rhob = 1 \oplus \omega_p^{(p-1)/2}$.
 The hypothesis $p \equiv 3 \pmod{4}$ means that $(p-1)/2$ is odd, and so $\rhob$ is odd and thus belongs to $R(1,p)$ ($\rhob$ is the representation attached to
 the Eisenstein series $E_k(z)$ where $k=1+(p-1)/2$ for $p>3$ and to $E_4(z)$ if $p=3$). Reasoning as in Proposition~\ref{recequi}, there is an eigenform $f$ in $\Fc(1,p)_{\rhob}$. If $M(1,p)$, hence $f$, has the equidistribution property, $\rhob(G_{\Q,p})-\{0\}$ generates $\F_p^\ast$ by Proposition~\ref{eigenequi}.
 Since the image of $\rhob$ is $\{0,2\}$, this implies that $2$ is a primitive root modulo $p$.
 \end{proof}
 
 \section{A variant: counting square-free integers with non-zero coefficients}
 \label{sectionvariant}
 
Given a modular form $f = \sum_{n=0}^\infty a_n q^n$ in $M(N,p)$, let 
$$
\pi_\sf(f,x) = |\{ n < x, n \text{ square-free}, a_n \neq 0\}|.
$$ 
Our proof of Theorem 1 allows us to get asymptotics for $\pi_{\sf}(f,x)$, and indeed this is a little simpler than Theorem 1.  
We state this asymptotic, and sketch the changes to our proof omitting details.  

 \begin{theorem} \label{mainvariant} If there exists a square-free integer $n$ with $a_n\neq 0$, then there exists a positive 
 real constant $c_\sf(f) >0$ such that
 $$
 \pi_\sf(f,x) \sim c_\sf(f) \frac{x}{(\log x)^{\alpha(f)}} (\log \log x)^{h(f)}.
 $$
 If $a_n=0$ for all square-free integers $n$, then in fact $a_n \neq 0$ only for those integers $n$ that are divisible by $\ell^2$ for some 
 prime $\ell$ dividing $Np$.  
 \end{theorem}


Suppose below that $f$ has some coefficient $a_n\neq 0$ with $n$ not divisible by the square of any prime dividing $Np$.  
 We first prove Theorem~\ref{mainvariant} for a pure form $f \in \Fc(N,p)$, as in \S\ref{proofasymppure}. In this case, our hypothesis on $f$ is equivalent to saying that $f$ is 
 non-constant. Then the proof given in \S\ref{proofasymppure} works  by replacing the sets $Z(f)$, $Z(f,a)$ by their intersection $Z_\sf(f)$, $Z_\sf(f,a)$
 with the set of square-free integers.  We have a  decomposition, analogous to~(\ref{Zfapart}) but simpler:  
\begin{equation}
 \label{Zfapartsf}  Z_\sf(f,a) = \coprod_{f',h} Z_\sf(f,a;f',h), 
 \end{equation}
where the disjoint union is taken over forms $f'$ in $Af-\{0\}$ and integers $0\le h\le h(f)$.   Here the set $Z_\sf(f,a;f',h)$ is 
defined as the set of integers $n=mm'$ with $(n,Np)=1$ such that   
\begin{num} \label{Zsnsf1} $m$ is square-free and all its prime factors are in $\Uc_f$; \end{num} 
\begin{num} \label{Zsnsf2} $m'$ is square-free, has exactly $h$ prime factors, and all its prime factors are in $\Nc_f$, and 
moreover $f'=T_{m'} f$; \end{num} 
\begin{num} \label{Zsnsf4} $a_m(f') = a$. 
\end{num}
The asymptotics for the number of integers $<x$ in $Z(f,a;f',h)$ is then exactly as in \S\ref{proofasymppure}, except that the set of square-full integers 
$\Sc_{f,f''}$ is now $\{1\}$. The desired asymptotics for $\pi_\sf(f)$ follows. 

The case where $f$ is in $\Fc(N,\F)$ but not necessarily pure is reduced to the pure case exactly as in~\S\ref{proofasymp2}.

Finally, in the general case where $f \in M(N,\F)$, let $\Bc_\sf$ be the set of {\it square-free integers $m$} whose prime factors all divide $Np$.
We observe that $\Bc_\sf$ is a finite subset of the infinite set $\Bc$ defined in \S\ref{pfasympgeneralcase}. For $m \in \Bc_\sf$, we define 
as in \S\ref{pfasympgeneralcase} $$
f_m = \sum_{\substack{ n = m m' \\ { (m',Np)=1}}} a_n q^n,
$$
and we have clearly
$$
\pi_\sf(f,x) = \sum_{m \in \Bc_\sf} \pi_\sf(f_m,x)
$$
By the assumption made on $f$, one of the $f_m$ for $m \in \Bc_\sf$ at least is non-constant. The rest of the proof is therefore exactly as in \S\ref{pfasympgeneralcase}.
 
 \section{Examples}
 
 \label{sectionexamples}
 
 \subsection{Examples in the case $N=1$, $p=3$}

 The simplest case where our theory applies is $N=1$, $p=3$. Let us denote by $\Delta = q + 2q^4+q^7+q^{13}+\dots \in \F_3[[q]]$ 
 the reduction mod 3 of the $q$-expansion of the usual $\Delta$ function. The space $M(1,\F_3)$ is the polynomial
 algebra in one variable $\F_3[ \Delta]$ and $\Fc(1,\F_3)$ is the subspace of basis $(\Delta^k)$ where $k$ runs among positive integers not divsible by $3$.
The set of Galois representations $R(1,\F_3)$ has only one element, $\rhob = 1 \oplus \omega_3$ where $\omega_3$ is the cyclotomic
 character mod $3$. Hence every non-zero form $f \in M(1,\F_3)$ is a generalized eigenform, and hence pure.  Thus 
 the sets $\Uc_f$, $\Nc_f$ are independent of $f$, and are respectively the set $\Uc$, $\Nc$ of prime numbers $\ell$ congruent to $1$, $2$ modulo $3$; and the invariant $\alpha(f)$ is 1/2.

 The invariant $h(f)$ is more subtle. Recall from \S\ref{puredef} that $h(f)$ is the largest integer $h$ such that there exists primes $\ell_1,\dots,\ell_h$ in 
 $\Nc_f$ (that is, congruent to  $2$ mod $3$) such that $T_{\ell_1} \dots T_{\ell_h} f \neq 0$. According to a result of Anna Medvedowski (see \cite{med}) 
$h(f)$ is also the largest $h$ such that $T_2^h f \neq 0$. Using this it is easy to compute the value of $h(\Delta^k)$ for small values of $k$, as shown below
(we omit the values of $k$ divisible by $3$ since $h(\Delta^{3k})=h(\Delta^k)$):
\par \medskip
\begin{tabular}{c||c|c|c|c|c|c|c|c|c|c|c|c|c}
$f$ & $\Delta$ & $\Delta^2$ & $\Delta^4$ & $\Delta^5$ & $\Delta^7$ & $\Delta^8$ & $\Delta^{10}$ & $\Delta^{11}$ & $\Delta^{13}$ & $\Delta^{14}$ & 
$\Delta^{16}$ & $\Delta^{17}$ & $\Delta^{19}$ \\
\hline
$h(f)$ & 0 & 1 & 2 & 3 & 4 & 5 & 4 & 5 & 4 & 5 & 4 & 5 & 6
\end{tabular}
\par \medskip
In general Medvedowski has shown ({\it loc. cit.}) that $h(\Delta^k) < 4 k^{\log 2/\log 3}$.  Numerical experiments suggest that perhaps $h(\Delta^k)$ is of the order $\sqrt{k}$ for large $k$ with $3\nmid k$, so there is perhaps some room to improve this upper bound (note ${\log 2/\log 3} \approx 0.63$).
 
 \par \bigskip
 
\subsubsection{Calculation of $\pi(\Delta^2,x)$} 

The invariant $c(f)$ is the most difficult to determine. We shall calculate $c(\Delta^2)$, illustrating the proof of our theorem in this simplest non-trivial case. To ease notations, set $f=\Delta^2$. The Hecke module $Af$ is a two-dimensional vector space 
generated by $f=\Delta^2$ and $\Delta$, and the Hecke algebra $A_f$ can be identified with the algebra of dual numbers $\F_3[\epsilon]$, where
$\epsilon \Delta^2=\Delta$ and $\epsilon \Delta=0$. The value of the operators $T_\ell$ and $\ell S_\ell$ in $A_f=\F_3[\epsilon]$ is given by the following table (cf. \cite[\S A.3.1]{BK}):
\par \medskip
\begin{tabular}{c || c | c | c | c} 
$\ell \pmod{9}$ & 1,4,7 & 2 & 5 & 8 \\
\hline 
$T_\ell$  & 2 & $\epsilon$ & $2 \epsilon$ & 0 \\
$\ell S_\ell$  & 1 & $- 1$ & $-1$ & $-1$
\end{tabular}
\par \medskip
\noindent From this, using \ref{Tnmult}, it is not difficult to compute $T_{\ell^n}$ for any $n$:
\par \medskip
\begin{tabular}{c || c|c|c|c|c|c|c|c|c|c|c|c|c} 
$\ell \pmod{9}$ & \multicolumn{3}{c|}{1,4,7} & \multicolumn{4}{c|}{2} & \multicolumn{4}{c|}{5}  & \multicolumn{2}{c}{8} \\

$n \pmod 6$ & 0,3  & 1,4&  2,5 & 0,2,4 & 1 & 3 & 5 & 0,2,4 & 1 & 3 & 5 &  0,2,4 & 1,3,5  \\
  \hline 
$T_{\ell^n}$ & 1 & 2 & 0 & 1 & $\epsilon$ & $2 \epsilon$ &0 & 1 & $2 \epsilon$ & $\epsilon$ &0 & 1 & 0 \\
\end{tabular}
 \par \medskip
 We are now ready to follow the proof of Theorem~\ref{main}. Since $f \in \Fc(1,\F_3)$ and $f$ is pure,  only \S\ref{proofasymppure} is relevant. 
 As in our analysis there, write $f = \sum_{n \geq 1} a_n q^n$ and
 for $a =1,2 \pmod{3}$, let $Z(f,a)$ be the set of integers $n$ such that $a_n =a$. The set $Z(f,a)$ is the disjoint union of sets $Z(f,a;f',f'',h)$ as in (\ref{Zfapart}),
 where $f',f'' \in Af-\{0\}$ and $h \leq h(f)=1$ is a non-negative integer. The subsets with $h=0$ have negligible contribution in view of~(\ref{densZfaffh}).
 When $h=1$, for the set $Z(f,a;f',f'',1)$ to be non-empty one must have $h(f'') = 1$ and $h(f')=0$, and since $f''$ and $f'$ must be the image of $f$
 by some Hecke operators, this implies in view of the table above that $f''$  is either $2 \Delta^2$ or $\Delta^2$, and $f'$ is either $2 \Delta$ or
 $\Delta$, so we have 4 sets $Z(f,a; f',f'',1)$ to consider for each value $1,2$ of $a$. As explained in \S\ref{proofasymppure}, to each permissible choice of $f'$, $f''$ is
attached a set $\Sc_{f,f''}$ of square-full integers, namely the set of square-full $m''$ such that $T_{m''} f = f''$, and a multi-frobenian set of height 1, that is, a frobenian set, $\Mc_{f',f''}$, which is the set of primes $\ell$ in $\Nc_f$ such that $T_\ell f'' = f'$. For every choice of $f''$, $f'$, one sees from the table above that $\Mc_{f',f''}$ is either
the set of primes congruent to $2 \pmod 9$, or to $5\pmod 9$,  and in any case $\delta(\Mc_{f',f''})=1/6$. 
The sets $\Sc_{f,f''}$ may be easily determined using our table above.  Thus $\Sc_{\Delta^2,\Delta^2}$ 
consists of square-full numbers where primes $\equiv 2 \pmod 3$ appear to an even exponent, an even number of primes $\equiv 1\pmod 3$ appear to exponents that are at least $2$ and 
$\equiv 1$ or $4 \pmod 6$, and other primes $\equiv 1 \pmod 3$ appear to exponents that are multiples of $3$.  The set $\Sc_{\Delta^2,2\Delta^2}$ consists of square-full numbers that are divisible by an odd number of primes $\equiv 1 \pmod 3$ appearing to exponents at least $2$ and $\equiv 1$ or $4\pmod 6$, other primes $\equiv 1 \pmod 3$ appearing to exponents that are multiples of $3$, and primes $\equiv 2\pmod 3$ appearing to even exponents.   
%

According to Theorem~\ref{delangemethod}, one has for $a=1$ or $2$, $f'=\Delta$ or $2 \Delta$, and $f''=\Delta^2$ or $2 \Delta^2$,
\begin{align} \label{Zfaff1}
 |\{n<x: n \in Z(f,a; f',f'',1)\} | 
\sim 
\Big(\sum_{s  \in \Sc_{f,f''}} \frac{C(\Uc,s)}{s}  \Big)
\Big( \frac{1}{6} \Big) \Big( \frac{1}{2} \Big) \frac{x}{(\log x)^{\frac{1}{2}}}\log \log x, 
\end{align}
where
$$
C(\Uc,s) = {C(\Uc)} \prod_{\ell \mid s \atop \ell \equiv 1 \pmod{3}} \Big(1+\frac{1}{\ell}\Big)^{-1}
$$
and 
\begin{align}
\label{constant}
C(\Uc) &= \frac{1}{\Gamma(\frac 12)} \prod_{p\equiv 1 \pmod 3} \Big(1+\frac 1p\Big)\Big(1-\frac 1p\Big)^{\frac 12} 
\prod_{p\not\equiv 1\pmod 3} \Big(1-\frac 1p \Big)^{\frac 12} \nonumber\\
& = 
\frac{\root 4 \of 3}{\pi\sqrt{2}} \prod_{p\equiv 1\pmod 3} \Big(1-\frac{1}{p^2}\Big)^{\frac 12} = 0.2913\ldots.  
\end{align} 
In (\ref{Zfaff1}), the factor $\frac{1}{6}$ is $\delta(\Mc_{f'',f'})$ and the factor $\frac{1}{2}$ is $\frac{|\Delta|}{|\Gamma|}$ (and this factor would disappear if we counted cases $a=1$ and $a=2$ together).



Adding up all the possibilities, using (\ref{Zfapart}), we finally obtain that 
$$
\pi(\Delta^2,x) \sim c(\Delta^2) \frac{x}{(\log x)^{\frac{1}{2}}}\log \log x,
$$ 
where
$$ 
c(\Delta^2)=\frac{1}{3} \sum_{s\in \Sc_{f,f}\cup \Sc_{f,2f}} \frac{C(\Uc,s)}{s} = 
\frac{C(\Uc)}{3} \prod_{\ell \equiv 1\pmod 3} \Big(1-\frac 1{\ell^3}\Big)^{-1} \prod_{\ell \equiv 2\pmod 3} \Big(1-\frac{1}{\ell^2}\Big)^{-1}.
$$

\subsubsection{Calculation of $\pi_\sf(\Delta^k,x)$ for $k=1,2,4,5,7,10$}

We describe the calculation of $c_\sf(\Delta^k)$ in these examples, which is simpler than evaluating $c(\Delta^k)$. For $h \geq 0$ an integer, 
let $\Mc_h$ be the set of integers that are the product of exactly $h$ distinct primes, all congruent to $2$ or $5$ modulo $9$.  This is a multi-frobenian set, attached to the cyclotomic extension $\Q(\mu_9)/\Q$ of Galois group $G=(\Z/9\Z)^\ast$, and one has $\delta(\Mc_h) = \frac{2^h}{ h! 6^h} = \frac{1}{h! 3^h}$. One can show that for $k=1,2,4,5,7,10$ and $h=h(\Delta^k)=0,1,2,3,4,4$ respectively,
and for $m' \in \Mc_h$, one has (with $f=\Delta^k$) that $T_{m'} f\neq 0$, and in fact 
$T_{m'} f = \Delta$ or $T_{m'} f = 2 \Delta$. 
  Also note that for $f' = \Delta$ or $f'=2 \Delta$, one also has $T_m f' = \Delta$ or $2 \Delta$ for any 
square-free $m$ with prime factors in $\Uc$, so that $a_m(f')\neq 0$. 
  
  Thus, the main contribution to 
 $Z_\sf(\Delta^k)$ is the set we call $\Zc(\Uc,\Mc_h,{1})$, namely the set of all square-free numbers $m m'$, where $m$ is any 
product of primes in $\Uc$ (i.e. congruent to $1 \pmod{3}$), and $m' \in \Mc_h$. According to our Theorem 6,
$$
\pi_\sf(\Delta^k,x) \sim \frac{C(\Uc)}{h! 3^h } \frac{x}{(\log x)^{1/2}}  (\log \log x)^h,\ \ k=1,2,4,5,7,10, 
$$
where $h=h(k)=0$, $1$, $2$, $3$, $4$, $4$ respectively, and $C(\Uc)$ is the constant appearing in \eqref{constant}.


 \subsection{Example of a non-pure form in the case $N=1$, $p=7$}

Examples of powers of $\Delta$ that are not pure arise $\pmod 7$.
There one has $\Delta^2 = f+ \Delta$, where $f=\Delta^2-\Delta$ is an eigenform for all the Hecke operators $T_\ell$ ($\ell$ a prime number with $\ell \neq 7$), with eigenvalue $\ell^2+\ell^3$. The Galois representation $\rhob_f$ corresponding to this system is $\omega_7^2 \oplus \omega_7^3$ where $\omega_7$ is the cyclotomic character modulo $7$. The set $\Nc_{\rhob_f}$ is the set of prime numbers $\ell$ that are congruent to $-1$ modulo $7$, and $\Uc_{\rhob_f}$ the set of prime numbers congruent to $1,2,3,4,5$ modulo $7$. One has $\alpha(f)=\alpha(\rhob_f)=1/6$. 

The form $\Delta$ is also of course an eigenform, with system of eigenvalues $\ell+\ell^4$ for $T_\ell$, corresponding to the Galois representation $\rhob_\Delta = \omega_7 \oplus \omega_7^4$ with $\alpha(\rhob_\Delta)  = 1/2$.

The decomposition $\Delta^2 = f + \Delta$ is thus the canonical decomposition into pure forms, and the pure form $\Delta$ can be neglected
because  $\alpha(\Delta)>\alpha(f)$. One finds
$$
\pi_\sf(\Delta^2,x)  \sim \pi_\sf(f,x)  \sim C(\Uc_{\rhob_f}) \frac{x}{(\log x)^{1/6}} 
$$
with
 $$
 C(\Uc_{\rhob_f}) = \frac{1}{\Gamma(5/6)} \prod_{\ell \equiv 1,2,3,4,5 \pmod{7} } \Big(1+\frac{1}{\ell}\Big) \Big(1-\frac 1{\ell}\Big)^{\frac 56} \prod_{\ell \equiv -1,0 \pmod{7}} \Big(1- \frac 1\ell\Big)^{\frac 56}
$$
so that 
$$
\pi_\sf(\Delta^2,x) \sim c_\sf(\Delta^2) \frac{x}{(\log x)^{1/6}},\ \ c_\sf(\Delta^2)=C(\Uc_{\rhob_f})=0.5976\ldots.
$$

\end{document}